\documentclass[11pt]{article}

\usepackage{layout}
\usepackage[makeroom]{cancel}
\usepackage{bbm}
\usepackage[affil-it]{authblk}

\usepackage{amsfonts}
\usepackage{amsmath}
\usepackage{amssymb}
\usepackage{amsthm}
\usepackage{mathtools}
\usepackage{graphicx} \usepackage{enumerate} \usepackage{multicol}
\usepackage{mathrsfs} \usepackage[all,cmtip]{xy}
\usepackage{enumerate}
\usepackage{cite}
\usepackage[dvipsnames]{xcolor}

\allowdisplaybreaks[4]

\newlength{\bibitemsep}
\newlength{\bibparskip}
\let\oldthebibliography\thebibliography
\renewcommand\thebibliography[1]{%
  \oldthebibliography{#1}%
  \setlength{\parskip}{\bibitemsep}%
  \setlength{\itemsep}{\bibparskip}%
}

\usepackage[colorlinks,linkcolor=blue,citecolor=blue,pagebackref,hypertexnames=false, breaklinks]{hyperref}

\newtheorem{theorem}{Theorem}[section]

\newtheorem{assumption}[theorem]{Assumption}
\newtheorem{definition}[theorem]{Definition}
\newtheorem{proposition}[theorem]{Proposition}

\newtheorem{corollary}[theorem]{Corollary}
\newtheorem{lemma}[theorem]{Lemma}
\newtheorem{remark}[theorem]{Remark}
\newtheorem{example}[theorem]{Example}
\newtheorem{examples}[theorem]{Examples}
\newtheorem{foo}[theorem]{Remarks}

\newenvironment{Example}{\begin{example}\rm}{\end{example}}

\def\vint{\mathop{\mathchoice%
          {\setbox0\hbox{$\displaystyle\intop$}\kern 0.22\wd0%
           \vcenter{\hrule width 0.6\wd0}\kern -0.82\wd0}%
          {\setbox0\hbox{$\textstyle\intop$}\kern 0.2\wd0%
           \vcenter{\hrule width 0.6\wd0}\kern -0.8\wd0}%
          {\setbox0\hbox{$\scriptstyle\intop$}\kern 0.2\wd0%
           \vcenter{\hrule width 0.6\wd0}\kern -0.8\wd0}%
          {\setbox0\hbox{$\scriptscriptstyle\intop$}\kern 0.2\wd0%
           \vcenter{\hrule width 0.6\wd0}\kern -0.8\wd0}}%
          \mathopen{}\int}

\newcommand{\R}{\mathbb R}

\newcommand{\pip}{\varphi}

\newcommand{\B}{\mathbf B}

\newcommand{\ve}{\varepsilon}

\newcommand{\eps}{\varepsilon}
\mathtoolsset{showonlyrefs=true}

\title{Besov class via heat semigroup on Dirichlet spaces II: BV functions and Gaussian heat kernel estimates}

\author{Patricia Alonso-Ruiz, Fabrice Baudoin,
Li Chen, Luke Rogers, Nageswari Shanmugalingam,
Alexander Teplyaev}

\date{}

\begin{document}

\maketitle

\begin{abstract}
We introduce the class of bounded variation (BV) functions in a general framework of strictly local Dirichlet spaces with doubling measure. Under the 2-Poincar\'e inequality and a weak Bakry-\'Emery curvature type condition, this BV class is identified with the heat semigroup based Besov class $\mathbf{B}^{1,1/2}(X)$ that was introduced in our previous paper. Assuming furthermore a quasi Bakry-\'Emery curvature type condition, we identify the Sobolev class $W^{1,p}(X)$ with $\mathbf{B}^{p,1/2}(X)$ for $p>1$. Consequences of those identifications in terms of isoperimetric and Sobolev inequalities with sharp exponents are given.

 \end{abstract}

\tableofcontents

\section{Introduction}

In a metric measure space $X$ that is highly path-connected, the theory of Sobolev classes based on upper gradients
provides an approach to calculus using a derivative structure that is strongly local. A weak upper gradient is 
an analog of $|\nabla f|$ when $f$ is a measurable function on the metric space; $|\nabla f|$ satisfies a variant of the 
fundamental theorem of calculus along most rectifiable curves in $X$, and it has the property that if $f$ is constant on a Borel set $E\subset X$,
then $|\nabla f|=0$ almost everywhere in $E$, see~\cite{HKST}. 
A corresponding approach to the theory of functions of bounded variation (BV) is also possible. Initially this was done in~\cite{Mr} under the assumption that the measure on $X$ is doubling and $X$ is sufficiently well-connected by paths to support a $1$-Poincar\'e inequality controlling $f$ by $|\nabla f|$, but it was subsequently recognized that weaker assumptions still ensure the existence of a rich theory~\cite{AmbrosioDiMarino}. In such a setting, a fruitful exploration of the geometry of $X$ using BV functions and sets of finite perimeter (sets whose characteristic functions are BV functions) is possible. 

However, there are metric measure spaces for which the preceding theory is degenerate~\cite[Corollary~7.5]{AmbrosioDiMarino}; a typical situation occurs on a fractal like the Sierpinski gasket, where a paucity of rectifiable curves (e.g.\ in the sense of $1$-modulus) causes the BV space coming from this approach to coincide with $L^1$.  
On the other hand, the theory of Dirichlet forms is well-developed in many such spaces, see for example \cite{Str03, KigB, ST, Str,ORS10, RT10, Ki93, Ki09, Kajino, Chee99, KRS, Koskela-Shanmugalingam-Tyson04}, which is far from an exhaustive list of the literature. The theory of regular Dirichlet forms assumes as the fundamental object a topological space equipped with a $\sigma$-finite Borel measure, and a closed non-negative definite quadratic form $\mathcal{E}$ with dense domain in the associated $L^2$-class on the space.

This paper is one of a sequence, including~\cite{ABCRST1,ABCRST3},  in which we define certain Besov spaces from a Dirichlet form and a measure, and use them as tools to explore and elucidate notions of BV.  The paper~\cite{ABCRST1} introduces some general theory about these spaces, while~\cite{ABCRST3} deals with some situations where the existing theory may degenerate as described above.  In this paper we examine the application of this approach in a metric upper gradient setting connected to that in~\cite{Mr,AmbrosioDiMarino,AES16}.  Specifically, we assume that the Dirichlet form provides an intrinsic metric with respect to which $\mu$ is a doubling measure, and where there is a $2$-Poincar\'e inequality involving the Dirichlet energy measures (Definition~\ref{def:Poincareineq}).  Under these conditions upper gradients exist and accordingly BV may be defined by the relaxation approach in~\cite{Mr}.   However, rather than assuming a $1$-Poincar\'e inequality as in~\cite{Mr}  we introduce a geometrically-motivated assumption that we call a weak Bakry-\'Emery condition (Definition~\ref{def:wBE}) and use it in establishing fundamental properties of BV. Some examples where one has $2$-Poincar\'e and weak Bakry-\'Emery but where the validity of a $1$-Poincar\'e inequality is unknown are in~\cite{BBG}.  Our approach can also be compared to the measure-valued upper gradient structure introduced in~\cite{AmbrosioDiMarino} and adapted to  Dirichlet spaces in~\cite{AES16}; we are indebted to the referee for pointing out that whether our approach lies in the scope of this theory depends on whether  $2$-Poincar\'e and our weak Bakry-\'Emery condition implies the $\tau$-regularity condition of~\cite[Definition~12.4]{AES16}.  The connection between the classical Bakry-\'Emery condition and Cheeger energy are discussed in some detail in~\cite[Section~10]{AES16}; we do not know whether $\tau$-regularity is true in our setting.  Moreover, we believe that the techniques used here may be of independent interest as tools in analyzing the BV space.

In a somewhat related direction to the present work,  Sobolev type spaces constructed using Dirichlet forms have been shown in~\cite{Koskela-Shanmugalingam-Tyson04} to coincide with those constructed using upper gradients if the metric space supports a $2$-Poincar\'e inequality. Moreover, from~\cite{Chee99} it follows that in a doubling metric measure space supporting a $p$-Poincar\'e inequality for some $1\le p\le 2$, there is a Dirichlet form that is compatible with the upper gradient Sobolev class structure, see for example~\cite{KRS}. 

In summary, the goal of this paper is to develop a theory of a BV class in the specific setting of a locally compact, separable topological space $X$,  equipped with a Radon measure $\mu$, a strictly local Dirichlet form $\mathcal{E}$ on $L^2(X,\mu)$ and its associated intrinsic metric $d_\mathcal{E}$, such that $\mu$ is $d_\mathcal{E}$-doubling and there is a $2$-Poincar\'e inequality.  The background and assumptions are established in Section~2.  In Section~3 we propose a notion of BV functions and sets of finite perimeter and prove several fundamental properties following the approach of~\cite{Mr}: these include the Radon measure property of the BV energy seminorm and a co-area formula connecting sets of finite perimeter to BV energy (see Theorem~\ref{lem:Co-area}).
Section~4 is the heart of the paper. We begin by 
comparing the heat semigroup-based Besov class $\mathbf{B}^{p,\alpha/2}(X)$ introduced in~\cite{ABCRST1} to a 
more classical Besov class $B^\alpha_{p,\infty}(X)$ that was defined in~\cite{GKS} from the intrinsic metric $d_\mathcal{E}$ rather than the heat semigroup $P_t$. Under the standing
assumptions ($\mu$ is doubling and $2$-Poincar\'e inequality) we show that $\mathbf{B}^{p,\alpha/2}(X)$ coincides
with $B^\alpha_{p,\infty}(X)$. 
This result is related to the correspondence of metric and heat semigroup  
Besov classes established in~\cite{Pietrushka-Paluba}, but differs in that the latter makes the stronger assumption that $\mu$ is Ahlfors
regular. We then compare  the class $BV(X)$ to
the heat semi-group Besov class $\mathbf{B}^{1,1/2}(X)$ and show that these coincide under the additional
hypothesis that $\mathcal{E}$ supports a weak Bakry-\'Emery curvature condition,
see Theorem~\ref{thm:W=BV}. We also explore a connection between the co-dimension~$1$ Hausdorff
measure of the regular boundary $\partial_\alpha E$ of a set $E$ of finite perimeter (meaning $\mathbf{1}_E\in BV(X)$)
to its perimeter measure $\| D\mathbf{1}_E\|(X)=:P(E,X)$, see Proposition~\ref{lem:Hausdorff-perimeter}. In the last part of  Section~4 we show that if $X$ supports a
quasi Bakry-\'Emery curvature condition and  $p>1$, then  the 
heat semigroup-based Besov class $\mathbf B^{p,1/2}(X)$ coincides with the Sobolev space $W^{1,p}(X)$, see 
Theorems~\ref{eq:Besov-Sobolev}, \ref{thm:BesovLB} and~\ref{thm:BesovUB}. 
Section~5 concludes with a discussion of Sobolev type embedding theorems for Besov and BV classes in the context of 
strictly local Dirichlet spaces satisfying the weak Bakry-\'Emery estimate.
These parallel the classical Sobolev embedding theorems associated with the classical Sobolev and BV classes as in~\cite{MazyaSobolev, AR}. 
The tools of heat semigroup based Besov spaces that will be used were developed in~\cite{ABCRST1}; nevertheless, the present paper can largely be read independently from the latter.

%In this chapter we define the space  $BV(X)$ of functions of 
%bounded variation when $\mathcal{E}$ is strcitly local (see the exact assumptions in Section 4.1). 
%We shall then prove that under a weak Bakry-\'Emery curvature dimension condition, 
%one actually has $BV(X)=\mathbf{B}^{1,1/2}(X)$ with comparable seminorms. One of the key tools used there is the co-area
%formula for BV functions; we establish this formula in Lemma~\ref{lem:Co-area} of the present chapter. We will also study  in that framework  the  spaces $\B^{p,\alpha}(X)$ introduced in Chapter 1 and relate them to Besov spaces previously considered in the literature. Finally, applications to Sobolev and isoperimetric inequalities are given.

%\subsection*{Acknowledgments} 
%\smallskip

\

\noindent{\bf Acknowledgments.} 
The authors thank Naotaka Kajino for many stimulating and helpful discussions.
The authors also thank the anonymous referee for comments that helped improve the exposition of the paper.
P.A-R.~was partly supported by the Feodor Lynen Fellowship, Alexander von Humboldt Foundation (Germany) the grant DMS \#1951577 and \#1855349 of the NSF (U.S.A.). 
F.B.~was partly supported by the grant DMS \#1660031 of the NSF (U.S.A.) 
and a Simons Foundation Collaboration grant.  
L.R.~was partly supported by the grant DMS \#1659643 of the NSF (U.S.A.).  
N.S.~was partly supported by the grants DMS \#1800161 and \#1500440 of the NSF (U.S.A.). 
A.T.~was partly supported by the grant DMS \#1613025 of the NSF (U.S.A.).

\section{Preliminaries}\label{sec:prelim}
\subsection{Strictly local Dirichlet spaces, doubling measures, and standing assumptions }\label{SubSec:StrictLocal}

Throughout the paper, $X$ will be a separable, locally compact topological  space equipped with a Radon measure $\mu$ supported on
$X$. Let $(\mathcal{E},\mathcal{F}=\mathbf{dom}(\mathcal{E}))$ be a Dirichlet form on $X$, meaning it is a 
densely defined,  closed,  symmetric and Markovian form on $L^2(X)$. The book~\cite{FOT} is a classical  
reference on the theory of Dirichlet forms. We also refer to the foundational papers by K.T. Sturm~\cite{St-I, St-II, St-III}.

We denote by $C_c(X)$ the vector space of all continuous functions with compact support in $X$ and $C_0(X)$ its 
closure with respect to the supremum norm. A core for $(X,\mu,\mathcal{E},\mathcal{F})$ is a subset 
$\mathcal{C}$ of $C_c(X) \cap \mathcal{F}$ which is dense in $C_c(X)$ in the supremum 
norm and dense in $\mathcal{F}$ in the $\mathcal{E}_1$-norm 
\begin{equation}\label{eq:Fnorm}
\|f\|_{\mathcal{E}_1} = \left( \| f \|_{L^2(X)}^2 + \mathcal{E}(f,f) \right)^{1/2}.
\end{equation}
The Dirichlet form $\mathcal{E}$ is called {\em regular} if it admits a core. It is \emph{strongly local} 
if for any two functions $u,v\in\mathcal{F}$ with compact supports such that $u$ is constant in a 
neighborhood of the support of $v$, we have $\mathcal{E}(u,v)=0$ (see~\cite[p. 6]{FOT}). 
We will assume that $(\mathcal{E},\mathcal{F})$ is a strongly local regular Dirichlet form on $L^2(X)$.  

Since $\mathcal{E}$ is regular, for every $u,v\in \mathcal F\cap L^{\infty}(X)$ we can define the energy measure $\Gamma (u,v)$ through the formula
 \[
\int_X\phi\, d\Gamma(u,v)=\frac{1}{2}[\mathcal{E}(\phi u,v)+\mathcal{E}(\phi v,u)-\mathcal{E}(\phi, uv)], \quad \phi\in \mathcal F \cap C_c(X).
\]
Then $\Gamma(u,v)$ can be extended to all $u,v\in \mathcal F$ by truncation (see \cite[Theorem 4.3.11]{ChenFukushima}). According to Beurling and Deny~\cite{Beurling-Deny}, one has then   for $u,v\in \mathcal{F}$
\[
\mathcal E(u,v)=\int_X d\Gamma(u,v)
\]
and $\Gamma(u,v)$ is a signed Radon measure. 

%\begin{definition}
Observe that the energy measures $\Gamma(u,v)$ inherit a strong locality property from $\mathcal E$, namely that $\mathbf 1_Ud\Gamma(u,v)=0$ for any open subset $U\subset X$ and $u,v\in \mathcal F$ such that $u$ is constant on $U$. One can then extend $\Gamma$ to $\mathcal{F}_{\mathrm{loc}}(X)$ defined as
\begin{equation}\label{def:local-F}
\mathcal{F}_{\mathrm{loc}}(X)=\{u\in L_{\mathrm{loc}}^2(X): \forall \text{ compact }K\subset X, \exists v\in \mathcal{F}\text{ such that } u=v|_K \text{ a.e.}\}.
\end{equation}
%\end{definition} 
We will still denote this extension by $\Gamma$ and collect below some of its properties for later use. For proofs, we refer for instance to~\cite[Section 3.2]{FOT} and also \cite[Section 4]{St-I}.
\begin{itemize}
\item {\em Strong locality.} For all $u,v\in \mathcal F_{\mathrm{loc}}(X)$ and all open subset $U\subset X$ on which $u$ is a constant
\[
\mathbf 1_Ud\Gamma(u,v)=0.
\]
\item {\em Leibniz and chain rules.} For all $u \in   \mathcal F_{\mathrm{loc}}(X) ,v\in \mathcal F_{\mathrm{loc}}(X) \cap L_{\mathrm{loc}}^{\infty}(X)$, $w\in  \mathcal F_{\mathrm{loc}}(X) $ and $\eta\in C^1(\mathbb R)$, we have $\eta(u)\in\mathcal{F}_{\text{loc}}$ and
\begin{align*}
d\Gamma(uv,w)&=ud\Gamma(v,w)+vd\Gamma(u,w),\\
d\Gamma(\eta(u),v)&=\eta'(u)d\Gamma(u,v).
\end{align*}
%{\color{red} (Here I think it is okay to have $u\in\mathcal{F}_{loc}(X)$ as long as $v\in \mathcal{F}_{loc}(X)\cap L^\infty(X)$, right?
%Nages)}
\end{itemize}

With respect to $\mathcal{E}$ one can define the following \emph{intrinsic metric} $d_{\mathcal{E}}$ on $X$ by
\begin{equation}\label{eq:intrinsicmetric}
d_{\mathcal{E}}(x,y)=\sup\{u(x)-u(y)\, :\, u\in\mathcal{F}\cap C_0(X)\text{ and } d\Gamma(u,u)\le d\mu\},
\end{equation}
where the condition $d\Gamma(u,u)\le d\mu$ means that $\Gamma(u,u)$ is absolutely continuous with 
respect to $\mu$ with Radon-Nikodym derivative bounded by $1$. 
%\note{notation convention: $d\Gamma(u,u)\le d\mu$ 
%or $\Gamma(u,u) \ll \mu$?}
%{\color{blue} By saying $\Gamma(u,u)\ll \mu$ we mean that $\Gamma(u,u)$ is absolutely continuous with respect to the
%measure $\mu$. Saying $d\Gamma(u,u)\le d\mu$ is much stronger, it says that for all Borel $A\subset X$ we have
%$\int_Ad\Gamma(u,u)\le \mu(A)$. Nages} \note{I was a little confused since we also write $\Gamma(u,u)\le K\mu$, which is $d\Gamma(u,u)\le K d\mu$. They are good for me now. Li}
%Here by $\Gamma(u,u)\le \mu$ we mean that for each Borel set $A\subset X$ we have $\Gamma(u,u)(A)=\int_Ad\Gamma(u,u)\le \mu(A)$.
The term ``intrinsic metric'' is potentially misleading because in
 general there is no reason why $d_{\mathcal{E}}$ is a metric on $X$ (it could be infinite for a given pair of points $x,y$
or zero for some distinct pair of points). %, however in this paper we will work in a standard setting in which it is a metric. The following definition is from
However, the setting we work in here will, by definition, rule out this possibility. Namely, we will assume the Dirichlet space to be \textit{strictly local} as defined e.g.\ in~\cite{Lenz11}, which is based on the classical papers \cite{St-II,St-III,St-I,BiroliMosco1,BiroliMosco2,BiroliMosco3}.

\begin{definition}
A strongly local  regular Dirichlet space  is called strictly local if $d_{\mathcal{E}}$ is a metric on $X$ and the topology induced by $d_{\mathcal{E}}$ coincides with the topology on $X$.
\end{definition}

%Strict locality will be assumed throughout the paper. 
% Let us stress that the equivalence between $d_{\mathcal{E}}$ and the original metric $d$ of the underlying space $X$ happens at the topological level. In fact, starting with the topology on $X$ generated by a metric $d$ on $X$, there is no reason why $d_{\mathcal{E}}$ should even be locally bi-Lipschitz equivalent to $d$, though for the Dirichlet form to be associated with $(X,d)$ one would like the topology generated by the two metrics to coincide. 
%For example, on the standard Sierpinski gasket equipped with the Kusuoka measure as a reference measure, the Euclidean metric $d$ generates the same topology as the  so-called harmonic shortest path metric $d_\mathcal{E}$ ~\cite{Kig-Sierpinski}. Whereas the standard Sierpinski gasket does not support  the necessary  Poincar\'e inequality with respect to the Euclidean metric $d$ (although it supports some other Poincar\'e inequalities \cite[and references therein]{KZ,OS-CT}), it does support a $2$-Poincar\'e inequality with respect to $d_{\mathcal{E}}$, see for example~\cite[Proposition~4.26]{Ka13r} and \cite{Koskela-Zhou}.

\begin{Example}\label{Cheeger differential structure}
In the context of a complete metric measure space $(X,d,\mu)$  supporting a $2$-Poincar\'e inequality and where 
$\mu$ is doubling, one can construct a Dirichlet form $\mathcal{E}$ with domain $N^{1,2}(X)$ by using a
choice of a Cheeger differential structure as in~\cite{Chee99}. This Dirichlet form is then strictly local and the intrinsic distance 
$d_\mathcal{E}$ is bi-Lipschitz equivalent to the original metric $d$. We refer to \cite{MMS} and the references 
therein for further details. This framework encompasses for instance the one of Riemannian manifolds with 
non-negative Ricci curvature and the one of doubling sub-Riemannian spaces   supporting a $2$-Poincar\'e inequality.
\end{Example}

\begin{Example}\label{ex:fractalshaveintrinsic form}
In the context of fractals, 
strictly local Dirichlet forms appear in  
\cite{Ka12r,Ka13r,Kig08,Kig93,ARKT,HT13,T04,MST,T08,FST,Koskela-Zhou,BiroliMosco1,BiroliMosco2} 
and play an important role in analysis of first-order derivatives in these settings. Whether every local 
Dirichlet form admits a change of measure under which it becomes  strictly local is an open question, though some natural   conditions for this are discussed in~\cite{HT,HKT15}, where it is also proved that $\Gamma$ is the norm of a well defined gradient that may be extended to measurable 1-forms, see~\cite{HRT}.
Without giving details of this analysis, we mention that 
existence of a suitable collection of finite (Dirichlet) energy coordinate functions, which depend only on the Dirichlet form $\mathcal{E}$,
is essentially equivalent to the existence of a measure 
which is compatible with an intrinsic distance. In particular, \cite{HKT15} 
proves existence of a measure which is compatible with an intrinsic distance 
for any  local resistance form  in the sense of Kigami~\cite{KigB,Kig93,Kig:RFQS,Ki09}. 
Thus, any fractal space with a local resistance form has an intrinsic metric and is a strictly local Dirichlet 
form for an appropriate choice of the measure.    
\end{Example}

Now suppose in addition to strict locality we know that open balls have compact closures and that $(X,d_\mathcal{E})$ is complete.  In this setting we may apply~\cite[Lemma~1, Lemma~$1^\prime$]{St-I} to obtain that the distance function $\varphi_x: y\mapsto d_{\mathcal{E}}(x,y)$ on $ X$ is in $\mathcal{F}_{\mathrm{loc}}(X)\cap \mathcal C$ and $d\Gamma(\varphi_x,\varphi_x)\le d\mu$. Then cut-off functions  on intrinsic balls $B(x,r)$ of the form 
\[
\varphi_{x,r}: y\mapsto (r-d_{\mathcal{E}}(x,y))_+
\]
are also in $\mathcal{F}_{\mathrm{loc}}(X)\cap \mathcal C$ and $d\Gamma(\varphi_{x,r},\varphi_{x,r})\le d\mu$ (for all
$r>0$ and $x\in X$). The following lemma will be useful. Its proof will be omitted because it is of a standard type related to that of the McShane extension theorem, see~\cite{Hei}, using approximations
of the form $f_j(x):=\inf\{f(q_i)+K\, \varphi_i(x)\, :\, i\in I\text{ with }i\le j\}$ with $\{q_i\, :\, i\in\mathbb{N}\}$ a countable
dense subset of $X$. % gives the following lemma.}

\begin{lemma}\label{lem:diff-of-Lip}
Let $f:X\to\R$ be locally Lipschitz continuous with respect to $d_{\mathcal{E}}$. Then $f\in\mathcal{F}_{\mathrm{loc}}(X)$
with $\Gamma(f,f)\ll\mu$. If $f$ is locally $K$-Lipschitz, then $\Gamma(f,f)\le K^2\, \mu$. 
\end{lemma} 

%\begin{proof}
%Let $Q$ be a countable dense subset of $X$.
%Let $U\subset X$ be a bounded open set and let $\{q_i\}_{i\in I\subset\mathbb{N}}$ be an enumeration of $Q\cap U$.
%Note that $Q\cap U$ is dense in $U$. For each $i\in I$ let $\psi_i(x)=d_{\mathcal{E}}(x,q_i)$. Then as explained above,
%$\psi_i\in\mathcal{F}(U)$ with $\Gamma(\psi_i,\psi_i)\le \mu$. For $j\in I$ set
%\[
%f_j(x):=\inf\{f(q_i)+K\, \psi_i(x)\, :\, i\in I\text{ with }i\le j\},
%\]
%where $K\ge 0$ is the Lipschitz constant of $f$ in $U$. The above functions are inspired by the proof of the McShane extension
%theorem (see for example~\cite{Hei}). By the lattice properties of Dirichlet forms, it is seen that each 
%$f_j\in\mathcal{F}(U)$ with $d\Gamma(f_j,f_j)\le K^2\, d\mu$. 
%Here, by the lattice property, we mean that if $u,v\in\mathcal{F}$, then $w_1=\min\{u,v\}$ and $w_2=\max\{u,v\}$ are also
%in $\mathcal{F}$ with $\Gamma(w_1,w_1)=\mathbf{1}_{\{u>v\}}\Gamma(v,v)+\mathbf{1}_{\{u\le v\}}\Gamma(u,u)$.
%Furthermore, $f_j$ are $K$-Lipschitz in $U$ with $f_j(q_i)=f(q_i)$ for
%$i\in I$ with $i\le j$.  We can see that $f_j\to f$ monotonically and hence (as $f$ and $f_j$ are bounded in $U$ because $U$ is bounded)
%$f_j\to f$ in $L^2(U)$, with $d\Gamma(f,f)/d\mu\le K^2$ on $U$. 
%\end{proof}

At many places in the paper we will need to approximate using locally Lipschitz functions and use locally Lipschitz cutoffs, so
it is important that these functions are dense in $L^1(X)$.  This is a consequence of density of simple functions and the fact that $\mu$ is Radon, for example by using the observation that if $K$ is compact and $U\supset K$ is open with $\mu(U\setminus K)<\eps$ then $\bigl\| \mathbf{1}_K - (1-d(x,K)/d(U^c,K))\bigr\|_{L^1}<\eps$ is a Lipschitz approximation of $\mathbf{1}_K$.

Now we come to the final assumption which will be made throughout the paper, namely that $\mu$ is volume doubling.
\begin{definition}
We say that the metric measure space $(X,d_{\mathcal{E}},\mu)$ satisfies the volume doubling property if 
there exists a constant $C>0$ such that for every $x\in X$ and $r>0$,
\[
\mu(B(x,2r))\le C\, \mu(B(x,r)).
\]
\end{definition}

It follows from the doubling property of $\mu$ (see~\cite{Hei}) that there is a constant $0<Q\ <\infty$ and $C\ge 1$ such that 
whenever $0<r\le R$ and $x\in X$, we have 
\begin{equation}\label{eq:mass-bounds}
 \frac{\mu(B(x,R))}{\mu(B(x,r))}
 \le C\left(\frac{R}{r}\right)^Q.
\end{equation}
%{\color{red}{Does it require $X$ to be unbounded for the upper bound ?}}

Another well-known consequence of the doubling property is the availability of a maximally separated $\eps$-covering as defined below. 
\begin{definition}\label{def:maxsepcover}
Let $U\subset X$ be a non-empty  subset and let $\eps>0$. A maximally separated $\eps$-covering is a family of  balls $\{B_i^\eps=B(x_i^\eps,\eps)\}_i$ such that 
\begin{itemize}
\item The collection $\{B_i^{\eps/2}\}_i$ is  a maximal pairwise disjoint family of balls with radius $\eps/2$;
\item The collection $\{B_i^\eps\}_i$ covers $U$, that is, $U=\bigcup_i B_i^\eps$;
\item For any $C>1$ there exists $K=K(C)\in \mathbb N$ such that each point $x\in X$ is contained in at most $K$ 
balls from the family $\{B_i^{C\eps}\}_i$.
\end{itemize}
\end{definition}

We now summarize the assumptions that will be in force throughout the paper.
\begin{assumption}\label{ass:I}
\ 
\begin{itemize}
\item The Dirichlet space $(X,\mu,\mathcal{E},\mathcal{F})$ is strictly local, so $\mathcal{E}$ is strongly local and regular and 
$d_{\mathcal{E}}$ is a metric on $X$ that induces the topology on $X$;
\item the metric space $(X,d_{\mathcal{E}})$ is complete;
\item $\mu$ is volume doubling;
%\item closed bounded subsets of $(X,d_{\mathcal{E}})$ are compact (this is a consequence of the preceeding two assumptions);
%\item locally Lipschitz functions are dense in $L^1(X)$.
%\item if $\Gamma(f,f)$ is absolutely continuous with respect to $\mu$, as is the case for locally Lipschitz functions, then $|\nabla f|$ is the square root of its Radon-Nikodym derivative, so  $\Gamma(f,f)=|\nabla f|^2d\mu$.
\end{itemize}
\end{assumption}

We will use the following consequences of these assumptions without further comment: closed and bounded subsets of $(X,d_{\mathcal{E}})$ are compact, and locally Lipschitz functions are dense in $L^1$. 
A commonly used notation throughtout the paper is that if $\Gamma(f,f)$ is absolutely continuous with respect to $\mu$, as is the case for locally Lipschitz functions, then $|\nabla f|$ denotes the square root of its Radon-Nikodym derivative, so  $d\Gamma(f,f)=|\nabla f|^2d\mu$. 
%Although it is shown in \cite{HRT} that   there exists a \mbox{$\mu$-a.e.~defined} 
%measurable gradient $\nabla f$, we do not use it in our paper.  

It should be noted that with the exception of some parts of Section~3, we will typically also assume existence of a 
$2$-Poincar\'e inequality, which is discussed next.

%\begin{assumption}\label{ass:II}
%The class of locally Lipschitz continuous functions on $X$ forms a dense subclass of $L^1(X)$. Moreover,
%we always assume that $(X,\mu,\mathcal{E},\mathcal{F})$ satisfies  the volume doubling property. 
%Apart from Section~3 where we define the BV class, we will also assume the $2$-Poincar\'e inequality. See
%Subsection~\ref{sec:D-P} for the definition of doubling property and Poincar\'e inequality.
%\end{assumption}

\subsection{The  2-Poincar\'e inequality} \label{sec:D-P}
Let $(X,\mu,\mathcal{E},\mathcal{F})$ be a strictly local regular Dirichlet space as in Section \ref{SubSec:StrictLocal}. 
%Throughout the paper, we always assume that $(X,\mu,\mathcal{E},\mathcal{F})$ satisfies  the volume doubling 
%property. Apart from Section 3 where we define the BV class, we will also assume the $2$-Poincar\'e inequality.

%For convenience of the reader, we recall  their definitions and several important consequences which will be useful in the sequel.

%Let $(X,\mu)$ be a topological measure space such that $\mu$ is a Radon measure on $X$. We also have $X$ 
%equipped with a Dirichlet form $\mathcal{E}$, that is, a symmetric bilinear form $\mathcal{E}:L^2(X)\times L^2(X)\to [-\infty,\infty]$, and
%we set the domain $\mathcal{F}$ of the Dirichlet form to be the class of all functions $u\in L^2(X)$ for which
%$\mathcal{E}(u,u)<\infty$. For $u\in\mathcal{F}$ and $t>0$ we will have $u_t:=\min\{u,t\}\in\mathcal{F}$ with
%$0\le \mathcal{E}(u_t,u_t)\le \mathcal{E}(u,u)$, that is, $\mathcal{E}$ is Markovian. We also require $\mathcal{F}$ to be
%dense in $L^2(X)$.

\begin{definition}\label{def:Poincareineq}
We say that $(X,\mu,\mathcal{E},\mathcal{F})$ supports the $2$-Poincar\'e
inequality if there are constants $C>0$ and $\lambda\ge 1$ such that 
whenever $B$ is a ball in $X$ and $u\in\mathcal{F}$, we have
\[
\frac{1}{\mu(B)}\int_B|u-u_B|\, d\mu\le C\, \mathrm{rad}\,(B)\, \left(\frac{1}{\mu(\lambda B)}\, \int_{\lambda B}d\Gamma(u,u)\right)^{1/2}.
\]
where $\mathrm{rad}\,(B)$ is the radius of $B$ and $\lambda B$ denotes a concentric ball of radius $\lambda\mathrm{rad}\,(B)$.
\end{definition}
\begin{remark}
 The $2$-Poincar\'e inequality does not need
$\mathcal{E}$ to be strictly local, but it does need it to be regular, in order for the measure $\Gamma(u,u)$ 
representing the Dirichlet energy of $u\in\mathcal{F}$ to exist, see~\cite{FOT} for more details.  However we will always be considering strictly local forms.
\end{remark}

\begin{example}
\em Examples of strictly local Dirichlet spaces $(X,\mu,\mathcal{E},\mathcal{F})$ that satisfy  the volume doubling property and support the 2-Poincar\'e inequality include:
\begin{itemize}
\item Complete Riemannian manifolds with non-negative Ricci curvature or more generally 
RCD$(0,N)$ spaces in the sense of Ambrosio-Gigli-Savar\'e \cite{AGS14};
\item Carnot groups and other complete sub-Riemannian manifolds satisfying a generalized curvature dimension inequality (see \cite{BB,BK14});
\item Doubling metric measure spaces that support a $2$-Poincar\'e inequality with respect to the upper gradient
structure of Heinonen and Koskela (see~\cite{HKST,Koskela-Shanmugalingam-Zhou,Koskela-Zhou});
\item Metric graphs with bounded geometry (see \cite{Ha11}).
\end{itemize}
\end{example}

When the 2-Poincar\'e inequality is satisfied, a standard argument due to Semmes tells us that locally 
Lipschitz continuous functions
form a dense subclass of $\mathcal{F}$, where $\mathcal{F}$ is equipped with the norm~\eqref{eq:Fnorm},
see for example~\cite[Theorem 8.2.1]{HKST}. Moreover, by~\cite{Koskela-Shanmugalingam-Tyson04}, we know that if 
the $2$-Poincar\'e inequality is satisfied and $\mu$ is doubling, then the 
Newton-Sobolev class (based on upper gradients, see
also~\cite{Koskela-Shanmugalingam-Tyson04}) is the same as the class $\mathcal{F}$, with
comparable energy seminorms. 

%{\color{Emerald} It would also be good to include some discussions on compactly supported functions $Lip_c$(X), we shall use the facts that $Lip_c$ is dense in $L^{\infty}$  in Lemma \ref{lem:L1-norm-control} and $Lip_c$ is dense in $\mathcal F$  in Lemma \ref{lem:HSminusF1}.} {\color{red} I am not sure that $Lip_c$ is dense in $L^\infty$. For example, the constant function $1$ is in
%$L^\infty$, but then the $L^\infty$-distance between $1$ and a compactly supported function is at least $1$ if $X$ is itself
%not compact, right? Nages}

The next lemma is used to define a \textit{length of the gradient} in the current setting and shows that the Dirichlet  form admits a carr\'e du champ operator.  In particular, the quantity $|\nabla u|$ is an upper gradient of $u$ in the sense of~\cite{Koskela-Shanmugalingam-Tyson04}, and it follows that $u\in\mathcal{F}$ satisfies some a-priori stronger Poincar\'e inequalities in which the integrability exponent for $|u-u_B|$ in Definition~\ref{def:Poincareineq} is higher, see~\cite{HK00}.

\begin{lemma}\label{absolute continuity}
Suppose that $(X,\mu,\mathcal{E},\mathcal{F})$ satisfies the doubling property and supports the $2$-Poincar\'e
inequality. Then for all $u\in\mathcal{F}$, we have $d\Gamma(u,u)\ll\mu$. The Radon-Nikodym derivative
$\tfrac{d\Gamma(u,u)}{d\mu}$ is denoted by $|\nabla u|^2$.
\end{lemma}

%{\bf The proof of this lemma is via a discrete convolution approximation of $u\in\mathcal{F}$; the
%convolutions are via the maximal $\varepsilon$-separated coverings of $X$ as explained above, and give
%locally Lipschitz (with respect to the metric $d_{\mathcal{E}}$) functions as approximations. Such arguments will be used in the proof of Theorem \ref{thm:BesovLB} below, 
%%so we leave the proof of the lemma to the interested reader, 
%see  also~\cite{Koskela-Shanmugalingam-Tyson04}.}

\begin{proof}
It is known~\cite[Lemma~2.1]{Hino2017} that if $u\in \mathcal{F}$ is the $\mathcal{E}_1$-limit (i.e.\ limit in the norm~\eqref{eq:Fnorm}) of functions $v_j\in\mathcal{F}$ satisfying $d\Gamma(v_j,v_j)\ll\mu$ then also $d\Gamma(u,u)\ll\mu$. By the assumed regularity of the Dirichlet form, we may therefore assume $u\in\mathcal{F}\cap C_c(X)$ and proceed to show that $u$ is such an  $\mathcal{E}_1$-limit.

Fix $\eps>0$. Let $\{B_i^\eps=B(x_i^\eps,\eps)\}_i$ be a maximally separated $\eps$-covering of $X$ as in Definition~\ref{def:maxsepcover}, so that   the family $\{B_i^{C\eps}\}_i$ has the bounded overlap property for any $C>1$. Let $\pip_i^\eps$ be a $(C/\eps)$-Lipschitz partition of unity subordinated to this cover: that is, 
$0\le \pip_i^\eps\le 1$ on $X$, $\sum_i\pip_i^\eps=1$ on $X$, and $\pip_i^\eps=0$ in $X\setminus B_i^{2\eps}$. We then
set 
\[
u_\eps:=\sum_i u_{B_i^\eps}\, \pip_i^\eps,
\]
where $u_{B_i^\eps}=\vint_{B_i^\eps} ud\mu$.

Since $u\in C_c(X)$ it is elementary that $u_\eps\to u$ in $L^2(X,\mu)$. Indeed, using the bounded overlap of the balls $B_i^{2\eps}$ we see that, as $\eps\to0$,
\begin{equation*}
\| u -u_\eps  \|_2^2 
\leq \sum_i \int_{B_i^{2\eps}} | u(x) -u_{B_i^\eps}|^2\,d\mu 
\leq  C\mu(\mathrm{sppt}(u))  \sup_i \sup_{x\in B_i^\eps} | u(x) -u_{B_i^\eps}|^2 \to 0.
\end{equation*}

Now each $\pip_i^\eps$ is Lipschitz, so we know that $u_\eps$ is locally Lipschitz and hence is in $\mathcal{F}_{\mathrm{loc}}(X)$. Indeed,
for $x,y\in B_j^\eps$ we have from the $2$-Poincar\'e inequality, the fact that $B_i^{2\eps}\cap B_j^{2\eps}\neq\emptyset$ implies $B_i^{2\eps}\subset B_j^{6\eps}$ and the volume doubling property, that
\begin{align*}
|u_\eps(x)-u_\eps(y)|
&\le \sum_{i:B_i^{2\eps}\cap B_j^{2\eps}\ne\emptyset}|u_{B_i^\eps}-u_{B_j^\eps}||\pip_i^\eps(x)-\pip_i^\eps(y)|\\
 &\le \frac{C\, d(x,y)}{\eps} \sum_{i:B_i^{2\eps}\cap B_j^{2\eps}\ne\emptyset}
 %    \left(\vint_{B_i^\eps}\vint_{B(x,2\eps)}|u(y)-u(x)|^2\, d\mu(y)\, d\mu(x)\right)^{1/2}\\
 \bigl( |u_{B_i^\eps} - u_{B_j^{6\eps}}| + |u_{B_j^\eps}- u_{B_j^{6\eps}}| \bigr) \\
 &\leq  \frac{C\, d(x,y)}{\eps} \vint_{B_j^{6\eps}} |u(y)-u_{B_j^{6\eps}}| \, d\mu(y) \\
&\le C\, d(x,y) \left(\frac{1}{\mu(B_i^{6\lambda\eps})} \int_{ B_j^{6\lambda\eps}}  d\Gamma(u,u)\right)^{1/2}. 
\end{align*}
It follows from Lemma \ref{lem:diff-of-Lip} that $\Gamma(u_\eps,u_\eps)\ll \mu$ and we recall that the Radon-Nikodym measure is denoted by 
$|\nabla u_{\eps}|^2$. Moreover, we also have on $B_i^\eps$ that 
\[
d\Gamma(u_\eps,u_\eps)\le C \left(\frac{1}{\mu(B_i^{6\lambda\eps})}\int_{\lambda B_i^{6\lambda\eps}}  d\Gamma(u,u)\right) d\mu.
\]
Using the doubling measure property again, and the bounded overlap of $\{B_i^{6\lambda\epsilon}\}$, this yields
\begin{align}\label{eq:Lip-Convergence}
\int_X |\nabla u_{\eps}|^2d\mu
=\mathcal E(u_\eps,u_\eps) 
\le \sum_i \int_{B_i^{\eps}} d\Gamma(u_\eps,u_\eps) 
\le C \sum_i \frac{\mu(B_i^{\eps})}{\mu( B_i^{6\lambda\eps})}\int_{ B_i^{6\lambda\eps}}  d\Gamma(u,u)\le C\mathcal E(u,u).
\end{align}

%In a similar manner, 
%\[
%|u(x)-u_\eps(x)|\le \sum_i |u(x)- u_{B_i^\eps}|\, |\pip_i^\eps (x)| \le \sum_i |u(x)- u_{B_i^\eps}|\, \mathbf 1_{B_i^{2\eps}}(x).
%\]
%Notice that the above sum has at most $K$ terms due to the finite overlap property, so that 
%\[
%\int_X|u(x)-u_\eps(x)|^2\, d\mu(x)
%\le C \sum_i \int_{B_i^{2\eps}} |u(x)- u_{B_i^\eps}|^2 d\mu.
%\]
%Now~\cite[Theorem~5.1.(1)]{HK00} gives $\int_B |u-u_B|^2\leq C\mathrm{rad}(B)^2 \int_{5\lambda B}\Gamma(u,u)$ from the $2$-Poincar\'e inequality; using this and the bounded overlap of $\{B_i^{10\lambda\epsilon}\}$ in the above we easily verify
%\[
%\int_X|u(x)-u_\eps(x)|^2\, d\mu(x)
%\le C  \eps^2 \sum_i  \int_{B_i^{10\lambda\eps}} d\Gamma(u,u)
%\le C \eps^2 \int_{X} d\Gamma(u,u).
%\]
%and thus $u_\eps\to u$ in $L^2(X)$ as $\eps\to 0^+$. 

Take a sequence $\eps_n\to 0^+$. From \eqref{eq:Lip-Convergence} and the  reflexivity of $\mathcal{F}$,
there exists a subsequence of $\{ u_{\eps_n}\}_n$ that is weakly convergent in $\mathcal{F}$. %$L^2(X)$. 
By Mazur's lemma, a sequence of convex combinations of $u_{\eps_n}$ (still denoted by $\{u_{\eps_n}\}$) converges in the $\mathcal{E}_1$-norm~\eqref{eq:Fnorm}.  Let the $\mathcal{E}_1$-limit of $\{u_{\eps_n}\}$ be denoted $v$. 
Since the $L^2$-norm is part of this norm and we know that $u_{\eps_n}\to u$ in $L^2(X)$, 
we have  $u=v$.   Thus we have $u_{\eps_n}\to u$ in $\mathcal{E}_1$-norm and $d\Gamma(u_{\eps_n},u_{\eps_n})\ll\mu$, which concludes the proof.
\end{proof}

\begin{definition}
Let $1\le p<\infty$. We say that $(X,\mu,\mathcal{E},\mathcal{F})$ supports a $p$-Poincar\'e inequality if there are constants
$C>0$ and $\lambda\ge 1$ such that whenever $B$ is a ball in $X$ and $u\in\mathcal{F}$, we have
\[
\frac{1}{\mu(B)}\int_B|u-u_B|\, d\mu\le C\, \mathrm{rad}(B)\, \left(\frac{1}{\mu(\lambda B)}\, \int_{\lambda B}|\nabla u|^p\, d\mu\right)^{1/p}.
\]
\end{definition}
Of course, the $p$-Poincar\'e inequality for any $p\ne 2$ does not make sense if $\mathcal{E}$
does not satisfy the condition of strict locality. %\Rd On the other hand, the $2$-Poincar\'e inequality does not need
%$\mathcal{E}$ to be strictly local, but it does need it to be regular (and hence the measure $\Gamma(u,u)$ that
%represents the Dirichlet energy of $u\in\mathcal{F}$ exists).\Bk
The requirement that $\mathcal{E}$ supports a $1$-Poincar\'e inequality is a significantly stronger requirement than supporting a $2$-Poincar\'e inequality.
%  \Rd On the other hand, if $\mathcal{E}$ is also strongly local,
%and in addition $\mathcal{E}$ is regular with  $\mu$ doubling with respect to $d_\mathcal{E}$ (as we assume in this paper as well),
%then $\Gamma(u,u)\ll\mu$ whenever $u\in\mathcal{F}$. \Bk

Much of the current
theory on functions of bounded variation in the metric setting requires a
$1$-Poincar\'e inequality, though the theory can be constructed without this~\cite{AmbrosioDiMarino}. In this paper we will \emph{not} require that $(X,\mu,\mathcal{E},\mathcal{F})$ supports a $1$-Poincar\'e inequality but only the weaker $2$-Poincar\'e inequality.  However in some of our analysis we will need an additional requirement called the weak
Bakry-\'Emery curvature condition.

\subsection{Sobolev classes $W^{1,p}(X)$}

The theory of Sobolev spaces was first advanced in order to prove solvability of certain PDEs, see for example
 \cite{Evans,MazyaSobolev}. When $X$ is a Riemannian manifold, a function 
$f\in L^p(X)$ is said to be in the Sobolev class $W^{1,p}(X)$ if its distributional derivative is given by a 
vector-valued function $\nabla f\in L^p(X:\R^n)$. Extensions of this idea to sub-Riemannian spaces have been considered in \cite{GN96}. However, in more general metric spaces where the distributional theory of derivatives
(which relies on integration by parts) is unavailable, an alternate notion of derivatives needs to be found. 
Indeed, we do not need an alternative to $\nabla f$, as long as we have a substitute for $|\nabla f|$.
For metric spaces
$X$, Lipschitz functions $f:X\to \R$ have a natural such alternative, $\text{Lip} f$, given by
\[
\text{Lip} f(x):=\limsup_{r\to 0^+}\sup_{y\in B(x,r)}\frac{|f(y)-f(x)|}{r}.
\]
Other notions such as upper gradients and Haj\l asz gradients play this substitute role well, see for example~\cite{HKST}.
In the current paper we consider another possible notion of $|\nabla f|$ which has a more natural affinity to the heat 
semigroup and the Dirichlet form, as in Lemma~\ref{absolute continuity}. 
So, in this paper, our definition of $W^{1,p}(X)$, $p \ge 1$  is the following:
\begin{align}\label{definition:W1p}
W^{1,p}(X)=\left\{ u \in L^p(X)\cap\mathcal{F}_{\mathrm{loc}}(X)\, :\,  \Gamma(u,u)\ll\mu, |\nabla u|\in L^p(X) \right\}.
\end{align}
The norm on  $W^{1,p}(X)$ is then  given by 
\[
\| u \|_{W^{1,p}(X)}=\| u \|_{L^p(X)} + \|  |\nabla u| \|_{L^p(X)}.
\]
%we say that $u\in W^{1,p}(X)$ if 
%$u\in L^p(X)\cap\mathcal{F}_{loc}$ with $\Gamma(u,u)\ll\mu$ and $|\nabla u|\in L^p(X)$.
%
Note, in particular, that $W^{1,2}(X)=\mathcal{F}$.
In the context of Sobolev spaces, Besov function classes arise naturally in two ways. Given a Sobolev class
$W^{1,p}(\R^{n+1})$ and a bi-Lipschitz embedding of $\R^n$ into $\R^{n+1}$, there is a natural trace of 
functions in $W^{1,p}(\R^{n+1})$ to the embedded surface, and this trace belongs to a Besov class, see
for example \cite{JW1,JW2}. % (\note{we choose one of them?}).
Besov classes also arise via real interpolations of $L^p(\R^n)$ and $W^{1,p}(\R^n)$, see for example
 \cite{BennettSharpley,Trie}. In the present paper we will
relate Sobolev classes $W^{1,p}(X)$ to two types of Besov classes defined in our previous  paper~\cite{ABCRST1}, see
Theorems~\ref{thm:BesovLB},~\ref{thm:BesovUB}, and~\ref{eq:Besov-Sobolev}. One of these types of Besov classes is defined from the heat semigroup, while the other uses only the metric structure of $X$. We note that previous metric characterizations of Sobolev spaces in the presence of doubling and 2-Poincar\'e have been studied in \cite{DMSq19}.

%{\color{green} Is the above sufficient? Nages}

\subsection{Bakry-\'Emery curvature conditions}

%\noindent {\bf Standing assumptions for this paper:} 
%In this paper we will assume that $\mathcal{E}$ is a strictly local
%Markovian Dirichlet form that induces a metric $d=d_\mathcal{E}$ with respect to which $\mu$ is doubling and supports a 
%$2$-Poincar\'e inequality. We will also assume that the class of locally Lipschitz continuous functions on $X$ forms a dense
%subclass of $L^1(X)$ and that $X$ is complete with respect to $d_\mathcal{E}$. As a consequence of
%the doubling property of $\mu$ it follows that closed and bounded subsets of $X$ are compact. 
Let $\{P_{t}\}_{t\in[0,\infty)}$ denote the self-adjoint  semigroup of contractions on $L^2(X,\mu)$ associated with the 
Dirichlet space $(X,\mu,\mathcal{E},\mathcal{F})$ and $L$ the infinitesimal generator of $\{P_{t}\}_{t\in[0,\infty)}$. 
The semigroup $\{P_{t}\}_{t\in[0,\infty)}$ is referred to as the heat semigroup on $(X,\mu,\mathcal{E},\mathcal{F})$. 
For classical properties of $\{P_{t}\}_{t\in[0,\infty)}$, we refer to~\cite[Section 2.2]{ABCRST1}. In particular, it is known that that doubling property together with the 2-Poincar\'e inequality imply that the semigroup $\{P_{t}\}_{t\in[0,\infty)}$ is conservative, i.e. $P_t 1=1$.

The work of Sturm \cite{St-II,St-III} (see Saloff-Coste~\cite{Sal-Cos} and Grigor'yan \cite{Gri91} for earlier 
results on Riemannian manifolds) tells us that doubling property together with the 2-Poincar\'e inequality 
are equivalent to the property that the heat semigroup $P_t$ admits a heat kernel function $p_t(x,y)$ on 
$[0,\infty)\times X\times X$ for which there are constants $c_1, c_2, C >0$ such that
whenever $t>0$ and $x,y\in X$,
\begin{equation}\label{eq:heat-Gauss}
 \frac{1}{C}\, \frac{e^{-c_1 d(x,y)^2/t}}{\sqrt{\mu(B(x,\sqrt{t}))\mu(B(y,\sqrt{t}))}}\le p_t(x,y)
   \le C\, \frac{e^{-c_2 d(x,y)^2/t}}{\sqrt{\mu(B(x,\sqrt{t}))\mu(B(y,\sqrt{t}))}}.
\end{equation}
The above inequalities are called  \emph{Gaussian bounds} for the heat kernel. Due to the doubling property, 
one can equivalently rewrite the Gaussian bounds as:
\begin{equation}\label{eq:heat-Gauss2}
 \frac{1}{C}\, \frac{e^{-c_1 d(x,y)^2/t}}{\mu(B(x,\sqrt{t}))}\le p_t(x,y)
   \le C\, \frac{e^{-c_2 d(x,y)^2/t}}{\mu(B(x,\sqrt{t}))},
\end{equation}
for some different constants $c_1, c_2, C >0$.
The combination of the doubling property and the 2-Poincar\'e inequality also implies the following H\"older regularity of the heat kernel
\begin{equation*}\label{Holder}
|p_t(x,y)-p_t(z,y)|\le \left(\frac{d(x,z)}{\sqrt t}\right)^{\alpha} \frac{C}{\mu(B(y,\sqrt t))},
\end{equation*}
for some $C>0$, $\alpha \in (0,1)$, and all $x,y,z\in X$ (see for instance \cite{saloff2002}).  In some parts of this paper,  
we need a stronger  condition than H\"older regularity for the heat kernel, in which case we will use the following uniform Lipschitz continuity property. 
 
 \begin{definition}\label{def:wBE}
 We  say that the Dirichlet metric space $(X,\mathcal{E}, d_\mathcal{E},\mu)$ satisfies a weak Bakry-\'Emery curvature 
 condition if, whenever $u\in \mathcal{F}\cap L^\infty(X)$ and $t>0$,
\begin{equation}\label{eq:weak-BE}
\Vert |\nabla P_tu| \Vert_{L^\infty(X)}^2\le \frac{C}{t}\Vert u\Vert_{L^\infty(X)}^2.
\end{equation}
\end{definition}
 
We refer to \eqref{eq:weak-BE}  as a weak Bakry-\'Emery curvature condition because, in many settings, its validity is 
related to the existence of curvature lower bounds on the underlying space. 

\begin{example}
\em The weak Bakry-\'Emery curvature condition is satisfied in the following examples:
\begin{itemize}
\item Complete Riemannian manifolds with non-negative Ricci curvature and more generally, the 
$RCD(0,+\infty)$ spaces (see \cite{Jiang15});
\item Carnot groups (see \cite{BB2});
\item Complete sub-Riemannian manifolds with generalized non-negative Ricci curvature (see \cite{BB,BK14});
\item  On non-compact metric graphs with finite number of edges, the weak Bakry-\'Emery curvature condition has been proved to hold for $t \in (0,1]$  (see   \cite[Theorem  5.4]{BK}), and is conjectured to be true for all $t$. If the graph is moreover compact, the weak Bakry-\'Emery estimate holds for every $t>0$ \cite[Theorem  5.4]{BK}) .
%\item {\color{red}{Sierpi\'nski gasket, equipped with the harmonic metric and Kusuoka measure (\cite{Kig-Sierpinski, Ka13}). In this case the weak BE is still a conjecture as far as I know, Fabrice}}
%{\color{blue} {I am confused. I thought you had a paper with
%Daniel Kelleher that the gasket with the harmonic measure and metric does have even the strong BE condition? Nages}}{\color{red}{It was indeed the goal in \cite{BK} to prove the strong BE on the harmonic gasket, however we finally only were able to prove on it on graphs approximations, and unfortunately the constants we obtained in the strong BE for the graphs approximations blow up at the limit. Fabrice}}
%
\end{itemize}
Several statements equivalent to the weak Bakry-\'Emery curvature condition are given in~\cite[Theorem~1.2]{CJKS}.  
There are some metric measure spaces equipped with a doubling measure supporting a $2$-Poincar\'e inequality 
but without the above weak Bakry-\'Emery condition, see for example~\cite{KRS}.  It should also be noted that, 
in the setting of complete sub-Riemannian manifolds with generalized non-negative Ricci curvature in the sense 
of~\cite{BG},  while the weak Bakry-\'Emery curvature condition is known to be satisfied (see \cite{BB,BK14}), 
the $1$-Poincar\'e inequality is so far not known to hold, though the $2$-Poincar\'e 
inequality is known to be always  satisfied, see \cite{BBG}.
\end{example}

We will also sometimes need a condition that is stronger than \eqref{eq:weak-BE}.

\begin{definition}
 We say that the Dirichlet metric space $(X,\mathcal{E}, d_\mathcal{E},\mu)$ satisfies a quasi Bakry-\'Emery curvature condition if there exists a constant 
$C>0$ such that for every $u \in \mathcal{F}$ and $t \ge 0$ we have $\mu$ a.e.
\begin{equation}\label{eq:strong-BE}
 |\nabla P_t u|  \le C P_t |\nabla u|  .
 \end{equation}
\end{definition}

The quasi Bakry-\'Emery curvature condition  implies the weak one, as is demonstrated in the proof of Theorem 3.3 in~\cite{BK}. Examples where the quasi Bakry-\'Emery estimate is satisfied include: Riemannian manifolds with non negative Ricci curvature  (in that case $C=1$, see \cite{MR2142879}), some metric graphs like the Walsh spider (see \cite[Example 5.1]{BK} and also \cite[Theorem  5.4]{BK})), the Heisenberg group and more generally H-type groups (see \cite{BBBC, Eldredge}).

The quasi Bakry-\'Emery curvature condition, while stronger than the weak Bakry-\'Emery curvature condition~\eqref{eq:weak-BE}, 
does not explicitly consider any dimension. In fact, it is weaker than a standard condition called the
Bakry-\'Emery  condition $BE(0,\infty)$ in strongly local Dirichlet spaces, see~\cite[Definition 3.1]{MR3121635} and also~\cite{EKS,AES16}. The $BE(0,\infty)$ condition is said to be
satisfied if, in the weak sense of  Definition 3.1 in \cite{MR3121635},   the Bochner-inequality  
\[
\frac12\left[L\Gamma(f,f)-2\Gamma(f,L f)\right]
  \ge 0,
\]
holds, where $L$ is the infinitesimal generator associated with the Dirichlet form. Under some regularity condition, by~\cite[Corollary 3.5]{MR3121635} the latter implies the gradient bound  $|\nabla P_t u|  \le  P_t |\nabla u|$, i.e. \eqref{eq:strong-BE} with $C=1$.  Thus, in the setting of this paper, the quasi Bakry-\'Emery curvature condition~\eqref{eq:strong-BE} is indeed weaker than $BE(0,\infty)$.

% From Theorem 7 in \cite{EKS} the $BE(0,n)$ condition implies, under some regularity condition, the gradient bound  $|\nabla P_t u|  \le  P_t |\nabla u|$, i.e. \eqref{eq:strong-BE} with $C=1$. 

In the context of RCD spaces, the  relation between the conditions RCD$(0,\infty)$  and $BE(0,\infty)$ is discussed in detail in \cite{MR3990192},  \cite{AGS14} and \cite{MR3494239}.

\section{BV class and co-area formula}\label{subsec:BV}

In this section we use the Dirichlet form and the associated family $\Gamma(\cdot,\cdot)$ of measures to construct a BV class
of functions on $X$. To do so, we only need $\mu$ to be a doubling measure on $X$ for $d_\mathcal{E}$ and the class of 
locally Lipschitz functions to be dense in $L^1(X)$. So in this section we do \emph{not} need the 2-Poincar\'e inequality
nor do we need the weak Bakry-\'Emery curvature condition. In the second part of the section we prove a co-area 
formula for BV functions; such a co-area formula is highly useful in  understanding the structure of BV functions, 
and underscores the importance of studying sets of finite perimeter (sets whose characteristic functions are BV functions).

\subsection{BV class}

In this subsection we will construct a BV class based on Dirichlet forms. The motivation for this definition
comes from the work of Miranda~\cite{Mr}.   In particular, in the context of a doubling  metric measure space $(X,d,\mu)$ supporting a $1$-Poincar\'e inequality,  where the Dirichlet form is given in terms of a Cheeger differential structure 
(see Example~\ref{Cheeger differential structure}), 
the  construction of $BV(X)$ and $\| Df \|$ is exactly that in~\cite{Mr}. It is also proved in~\cite{Mr} that this construction yields the usual notion of variation when applied to Riemannian or sub-Riemannian spaces.

We  set the \emph{core} of the Dirichlet form,
$\mathcal{C}(X)$, to be the class of all $f\in\mathcal{F}_{\mathrm{loc}}(X)\cap C(X)$ such that $\Gamma(f,f)\ll \mu$
and recall that the Sobolev class $W^{1,1}(X)$ is  the class of all $f\in\mathcal{F}_{\mathrm{loc}}(X)\cap L^1(X)$ for which
$\Gamma(f,f)\ll\mu$ and $|\nabla f|\in L^1(X)$ (see Definition \eqref{definition:W1p}).
%\emph{core} of the Dirichlet form,
%$\mathcal{C}(X)$, to be the class of all $f\in\mathcal{F}_{loc}(X)$ such that $\Gamma(f,f)\ll \mu$.
\begin{definition}\label{def-BV}
We say that $u\in L^1(X)$ is in $BV(X)$ if there is a sequence of local Lipschitz functions $u_k\in L^1(X)$
such that $u_k\to u$ in $L^1(X)$ and 
\[
\liminf_{k\to\infty}\int_X|\nabla u_k|\, d\mu<\infty.
\]
\end{definition}
We note that if the Dirichlet form supports a $1$-Poincar\'e inequality, then  the Sobolev space $W^{1,1}(X)$ is a subspace of $BV(X)$.

\begin{definition}\label{def-Du}
For $u\in BV(X)$ and open sets $U\subset X$, we set
\[
\Vert Du\Vert(U)=\inf_{u_k\in\mathcal{C}(U), u_k\to u\text{ in }L^1(U)}\liminf_{k\to\infty} \int_U|\nabla u_k|\, d\mu.
\]
\end{definition}
We will see in the next part of this section that $\Vert Du\Vert$ can be extended from the collection of
open sets to the collection of all Borel sets as a Radon measure, see Definition~\ref{def:BVgeneral sets}.
The following lemmas are standard: the first  can proved by applying the Leibniz
rule to the approximations $\eta u_k+(1-\eta)v_k$ with $u_k$ and $v_k$ the sequences of functions from $\mathcal{F}$
that approximate $u$ and $v$, and the non-trivial part of the second is a partioning argument.

% and then for sets $A\subset X$ we define
%\[
%\Vert Du\Vert(A)=\inf\{\Vert Du\Vert(O)\, :\, A\subset O\text{ and }O\text{ is open in }X\}.
%\]
%\note{This definition for general sets is the same as $\Vert Du\Vert^*(A)$ in Definition \ref{def:BVgeneral sets}. }

\begin{lemma}\label{lem:BV-Leibniz}
If $u,v\in BV(X)$ and $\eta$ is a Lipschitz continuous function on $X$ with $0\le \eta\le 1$ on $X$, then
$\eta u+(1-\eta)v\in BV(X)$ with
\[
\Vert D(\eta u+(1-\eta)v)\Vert(X)\le \Vert Du\Vert(X)+\Vert Dv\Vert(X) +\int_X|u-v|\, |\nabla \eta|\, d\mu.
\]
%Here \Rd $g_\eta$ \Bk is the minimal $1$-weak upper gradient of $\eta$.
\end{lemma}

%\begin{proof}
%From Lemma~\ref{lem:diff-of-Lip} we already know that such $\eta$ are
%in $\mathcal{F}_{\mathrm{loc}}(X)$ with $|\nabla \eta|\in L^\infty(X)$.
%From the definition, we can choose sequences $u_k,v_k\in L^1(X)$ of locally Lipschitz functions on $X$ 
%such that $u_k\to u$ and $v_k\to v$ in $L^1(X)$ and
%$\int_X |\nabla u_k|\, d\mu\to \Vert Du\Vert(X)$ and $\int_X |\nabla v_k|\, d\mu\to \Vert Dv\Vert(X)$ as $k\to\infty$.
%Now an application of the Leibniz rule %(see~\cite[page~190]{St-I} as well as the references therein) 
%to the functions
%$\eta u_k+(1-\eta)v_k$ tells us that
%\begin{align*}
%\Vert D(\eta u+(1-\eta)v)\Vert(X)&\le \liminf_{k\to\infty}\int_X |\nabla[\eta u_k+(1-\eta)v_k]|\, d\mu\\
% & \le \liminf_{k\to \infty}\left(\int_X\eta |\nabla u_k|\, d\mu+\int_X(1-\eta) |\nabla v_k|\, d\mu
% +\int_X |u_k-v_k|\, |\nabla \eta|\, d\mu\right).
%\end{align*}
%Now using  $0\le \eta\le 1$ and $u_k-v_k\to u-v$ in $L^1(X)$ we obtain the required inequality.
%\end{proof}

%The following elementary properties of $\Vert Du\Vert$.

\begin{lemma}\label{lem:elementary}
Let $U$ and $V$ be two open subsets of $X$. If $u\in BV(X)$, then
\begin{enumerate}
\item $\Vert Du\Vert(\emptyset)=0$,
\item $\Vert Du\Vert(U)\le \Vert Du\Vert(V)$ if $U\subset V$,
\item $\Vert Du\Vert(\bigcup_i U_i)=\sum_i\Vert Du\Vert(U_i)$ if $\{U_i\}_i$ is a pairwise disjoint subfamily of open
subsets of $X$.
\end{enumerate}
\end{lemma}

We use the above definition of $\Vert Du\Vert$ on open sets in a Caratheodory construction.
\begin{definition}\label{def:BVgeneral sets}
For $A\subset X$, we set
\[
\Vert Du\Vert^*(A):=\inf\{\Vert Du\Vert(O)\, :\, O\text{ is an open subset of }X, A\subset O\}.
\]
\end{definition}

By Lemma~\ref{lem:elementary} $\Vert Du\Vert^*(A)=\Vert Du\Vert(A)$ when $A$ is open; abusing notation we re-name $\Vert Du\Vert^*(A)$ as $\Vert Du\Vert(A)$ for general $A$.

The main result of this section is that $\|Du\|$, as constructed above, is a Radon measure on $X$.
The proof may be directly adapted from that for~\cite[Theorem~3.4]{Mr}, and relies on a lemma of De Giorgi and Letta~\cite[Theorem~5.1]{DeLe}, see also~\cite[Theorem~1.53]{AFP}. 
%The principal tool used in the proof is the following lemma
%%;\subsection{Radon measure property of BV energy} this lemma is 
%due to De Giorgi and Letta~\cite[Theorem~5.1]{DeLe}, see 
%also~\cite[Theorem~1.53]{AFP}.
%
%
%\begin{lemma}[\protect{\cite[Theorem~5.1]{DeLe}}]\label{lem:DeGLett}
%If $\nu$ is a non-negative function on the class of all open subsets of $X$ such that for open sets $U_1$, $U_2$
%\begin{enumerate}
%\item $\nu(\emptyset)=0$, \label{DeLett1}
%\item if $U_1\subset U_2$  then $\nu(U_1)\le \nu(U_2)$, \label{DeLett2}
%\item $\nu(U_1\cup U_2)\le \nu(U_1)+\nu(U_2)$, \label{DeLett3}
%\item if $U_1\cap U_2$ is empty then $\nu(U_1\cup U_2)=\nu(U_1)+\nu(U_2)$, \label{DeLett4}
%\item $\nu(U_1)=\sup\{\nu(V)\, :\,  V\text{ is bounded and open in }X\text{ with }\overline{V}\subset U_1\}.$\label{DeLett5}
%\end{enumerate}
%Then the Carath\'eodory extension of $\nu$ to all subsets of $X$ is a Borel regular outer measure on $X$.
%\end{lemma}

%The proof of the following theorem is based on that of~\cite{Mr}.

\begin{theorem}\label{thm:outer-measure}
If %$X$ is  complete and
$f\in BV(X)$, then $\Vert Df\Vert$ is a Radon outer measure on $X$ and
the restriction of $\Vert Df\Vert$ to the Borel sigma algebra is a Radon measure which is the weak limit of $\Vert Du_k\Vert$
for some sequence  $u_k$ of locally Lipschitz functions in $L^1(X)$ such that $u_k\to f$ in $L^1(X)$. 
%from Definition~\ref{def-Du}. 
\end{theorem}

\begin{Example}
There is a large class of fractal examples 
\cite{T08,KigB,Ki93,Kus89} 
with resistance forms $\mathcal E$, 
a so-called Kusuoka measure $\mu$, 
and  a base of open sets $O$ with finite boundaries, 
such that $ 1_O\in BV(X)$ and $\| D1_O\|$ 
is absolutely continuous 
with respect to the counting measure on $\partial O$. 
Among these examples, the most notable are 
the Sierpinski gasket in harmonic coordinates  \cite{Ka12r,Ka13r,Kig08,T04,MST,FST,Koskela-Zhou} and
fractal quantum graphs \cite{ARKT}.
%and diamond fractals \cite[and references therein]{AlonsoRuiz2018}. In particular, on diamond fractals \cite{AlonsoRuiz2018} provides explicit formulas for the heat kernel, which allow for many computations relevant to our paper. 
On the Sierpinski gasket \cite[Proposition 4.14]{T04} shows how to make computations at the dense set of junction points. 
One might expect that if $u\in BV(X)$ then, following \cite{HRT,HT13}, 
$Du$ could be defined as a vector valued Borel measure, however the details of this construction are outside of the scope of this article. 
The long term motivation for this type of analysis comes from stochastic PDEs, see  % \cite[and references therein]
\cite{BarbuRoeckner,Khoshnevisan1,Khoshnevisan2,MunteanuRoeckner,HZ12} and the references therein. 
\end{Example}

%%%%%%%%%%%%%%%%%%%%%%%%%%%%%%%%%
%%%%%%%%%%%%%%%%%%%%%%%%%%%%%%%%

\subsection{Co-area formula}

We give a co-area formula that connects the BV energy seminorm of a BV function
with the perimeter measure of its super-level sets.

\begin{definition}
A function $u$ is said to be in $BV_{\mathrm{loc}}(X)$ if for each bounded open set $U\subset X$ there is a compactly supported
Lipschitz function $\eta_U$ on $X$ such that $\eta_U=1$ on $U$ and $\eta_U\, u\in BV(X)$.
We say that a measurable set $E\subset X$ is of \emph{finite perimeter} if $\mathbf{1}_E\in  BV_{\mathrm{loc}}(X)$ 
%\note{(definition of $ BV_{loc}(X)$)} with
with $\Vert D\mathbf{1}_E\Vert(X)<\infty$.  For any Borel set $A\subset X$, we denote by 
$P(E,A):=\Vert D\mathbf{1}_E\Vert(A)$ the perimeter measure of $E$.
\end{definition}

The proof of the following theorem is a direct adaptation of the corresponding result for the BV theory 
found in~\cite[Proposition~4.2]{Mr}. %We leave this to the interested reader.
\begin{theorem}\label{lem:Co-area}
The co-area formula holds true, that is, for Borel sets $A\subset X$ and $u\in L^1_{\mathrm{loc}}(X)$,
\[
\Vert Du\Vert(A)=\int_{\R}P(\{u>s\}, A)\, ds.
\]
\end{theorem}

\section{BV, Sobolev and  heat semigroup-based Besov classes }\label{weak BE metric space}

%In this section we will assume that $(X,d_\mathcal{E},\mu)$ is doubling, supports a $2$-Poincar\'e inequality, 
%and satisfies a weak Bakry-\'Emery curvature
%condition~\eqref{eq:weak-BE}. Under this condition, we will compare the BV class, constructed in 
%Section~\ref{subsec:BV},
%with the Besov class $\mathbf{B}^{1,1/2}(X)$ from \cite{ABCRST1}. %~\eqref{eq:def:Besov}.
%
Throughout the section, let $(X,\mu,\mathcal{E}, \mathcal{F}, d_\mathcal{E})$ be a strictly local regular Dirichlet space that satisfies the general assumptions of Section \ref{sec:prelim} and the 2-Poincar\'e inequality. We stress that the 1-Poincar\'e inequality is not assumed.

\subsection{Heat semigroup-based Besov classes}

We  first turn our attention to the study of Besov classes. In~\cite{ABCRST1} the following heat semigroup-based Besov classes were introduced.

\begin{definition}[\cite{ABCRST1}]\label{def:Besov}
Let $p \ge 1$ and $\alpha \ge 0$. For $f \in L^p(X)$, we define the Besov  seminorm:
\[
\| f \|_{p,\alpha}= \sup_{t >0} t^{-\alpha} \left( \int_X \int_X p_t(x,y) | f(x)-f(y)|^p d\mu(x)d\mu(y) \right)^{1/p},
\]
and the Besov  spaces
\begin{align}\label{eq:def:Besov}
\mathbf{B}^{p,\alpha}(X)=\{ f \in L^p(X)\, :\,  \| f \|_{p,\alpha} <+\infty \}.
\end{align}
\end{definition}
The norm on $\mathbf{B}^{p,\alpha}(X)$ is defined as:
\[
\| f \|_{\mathbf{B}^{p,\alpha}(X)} =\| f \|_{L^p(X)} + \| f \|_{p,\alpha}.
\]
It is proved in Proposition 4.14 and Corollary 4.16 of~\cite{ABCRST1}
that $\mathbf{B}^{p,\alpha}(X)$ is a Banach space for $p\ge 1$ and that it is reflexive for $p>1$.
 In this section, we compare the spaces $\mathbf{B}^{p,\alpha}(X)$ to more classical notions of 
 Besov classes that have previously been considered in the metric setting. 
 
 We recall the following definition from~\cite{GKS}. For $0\le\alpha<\infty$, $1\le p<\infty$
and $p<q\le \infty$, let $B^\alpha_{p,q}(X)$ be the collection of functions
$u\in L^p(X)$ for which, if $q<\infty$
\begin{equation}\label{eq:BesovMetric-q<infty}
\Vert u\Vert_{B^\alpha_{p,q}(X)}:=\left(\int_0^\infty
\left(\int_X\int_{B(x,t)}\frac{|u(y)-u(x)|^p}{t^{\alpha p}\mu(B(x,t))}\, d\mu(y)\, d\mu(x)
\right)^{q/p}\, \frac{dt}{t}\right)^{1/q} <\infty
\end{equation}
and in the case $q=\infty$
\begin{equation}\label{eq:BesovMetric-q=infty}
\Vert u\Vert_{B^\alpha_{p,\infty}(X)}:=\sup_{t>0}
\left(\int_X\int_{B(x,t)}\frac{|u(y)-u(x)|^p}{t^{\alpha p}\mu(B(x,t))}\, d\mu(y)\, d\mu(x)
\right)^{1/p}<\infty.
\end{equation}
 %We now show connecting the two previously defined Besov classes is the following:

\begin{proposition}\label{prop:B=B}
For $1\le p<\infty$ and $0< \alpha<\infty$ we have
\[
\mathbf{B}^{p,\alpha/2}(X)=B^{\alpha}_{p,\infty}(X),
\]
with equivalent seminorms.
\end{proposition}

\begin{proof}
Since $\mu$ is doubling and supports a $2$-Poincar\'e inequality, we have the
Gaussian double bounds \eqref{eq:heat-Gauss2} for $p_t(x,y)$. Hence if $u\in \mathbf{B}^{p,\alpha}(X)$, we then
must have
\begin{align*}
\Vert u\Vert_{p,\alpha/2}^p
&\ge C^{-1}\sup_{t>0}\int_X\int_X\frac{|u(y)-u(x)|^p}{t^{\alpha p/2}}\,
 \frac{e^{-c\, d(x,y)^2/t}}{\mu(B(x,\sqrt{t}))}\, d\mu(y)\, d\mu(x)\\
&\ge C^{-1}\sup_{\sqrt{t}>0}\int_X \int_{B(x,\sqrt{t})}
 \frac{|u(y)-u(x)|^p}{t^{\alpha p/2}}\,
 \frac{e^{-c\, d(x,y)^2/t}}{\mu(B(x,\sqrt{t}))}\, d\mu(y)\, d\mu(x)\\
&\ge C^{-1} \sup_{\sqrt{t}>0}\int_X \int_{B(x,\sqrt{t})}
 \frac{|u(y)-u(x)|^p}{t^{\alpha p/2}\, \mu(B(x,\sqrt{t}))}
 \, d\mu(y)\, d\mu(x)\\
&=C^{-1}\Vert u\Vert_{B^\alpha_{p,\infty}(X)}^p,
\end{align*}
and from this it follows that $\mathbf{B}^{p,\alpha/2}(X)$ embeds boundedly into
$B^\alpha_{p,\infty}(X)$.

Now we focus on proving the converse embedding. From \eqref{eq:mass-bounds} and \eqref{eq:heat-Gauss2}, we have
\begin{align*}
&\frac{1}{t^{\alpha p/2}}\int_X\int_X|u(y)-u(x)|^p\, p_t(x,y)\, d\mu(y)\, d\mu(x)\\
&\le \frac{C}{t^{\alpha p/2}}\int_X\sum_{i=-\infty}^\infty\int_{B(x,2^i\sqrt{t})\setminus B(x,2^{i-1}\sqrt{t})}
 \frac{|u(y)-u(x)|^p\, e^{-c4^i}}{\mu(B(x, \sqrt{t}))}\, d\mu(y)\, d\mu(x)\\
 &\le \frac{C}{t^{\alpha p/2}}\int_X\sum_{i=-\infty}^\infty\int_{B(x,2^i\sqrt{t})}
 \frac{|u(y)-u(x)|^p\, e^{-c4^i}}{\mu(B(x,2^i\sqrt{t}))}\, \frac{\mu(B(x,2^i\sqrt{t}))}{\mu(B(x,\sqrt{t}))}
 \, d\mu(y)\, d\mu(x)\\
 &\le \frac{C}{t^{\alpha p/2}}\sum_{i=-\infty}^\infty e^{-c 4^i}\,\max\{1,2^{iQ}\}\, (2^i\sqrt{t})^{\alpha p}\,
 \int_X\int_{B(x,2^i\sqrt{t})}\frac{|u(y)-u(x)|^p}{(2^i\sqrt{t})^{\alpha p}\, \mu(B(x,2^i\sqrt{t}))}\, d\mu(y)\, d\mu(x)\\
 &\le C\, \Vert u\Vert_{B^\alpha_{p,\infty}(X)}^p\, \sum_{i=-\infty}^\infty e^{-c 4^i}\, 2^{i\alpha p}\, \max\{1,2^{iQ}\}.
\end{align*}
Since 
\[
\sum_{i=-\infty}^\infty e^{-c 4^i}\, 2^{i\alpha p}\, \max\{1,2^{iQ}\}\le 
\sum_{i\in\mathbb{N}}e^{-c 4^i}\,  2^{i(\alpha p+Q)}\, +\, \sum_{i=0}^\infty 2^{-i\alpha p}<\infty,
\]
the desired bound follows.
\end{proof}

\subsection{Under the weak Bakry-\'Emery condition, $\mathbf{B}^{1,1/2}(X)=BV(X)$}

Recall from Definition~\ref{def:local-F} that $u\in\mathcal{F}_{\mathrm{loc}}(X)$ if
for each ball $B$ in $X$ there is a compactly supported Lipschitz function $\varphi$ with $\varphi=1$ on $B$
such that $u\varphi\in\mathcal{F}$; in this case we can set $|\nabla u|=|\nabla (u\varphi)|$ in $B$, thanks to  the strict locality 
property  of $\mathcal{E}$. 
%{\color{red} (We used $\mathcal{F}_{\mathrm{loc}}$ before without defining. Should we move this
%definition to earlier, say in Section~2? Nages)} {\color{blue}( $\mathcal{F}_{\mathrm{loc}}$ was defined in Section \ref{SubSec:StrictLocal} in a different form. It should coincide with this definition here. )}

\begin{lemma}\label{lem:L1-norm-control}
Suppose that the weak Bakry-\'Emery condition~\eqref{eq:weak-BE} holds.
Then for $u\in \mathcal{F}\cap W^{1,1}(X)$, we have that
\[
\Vert P_t u-u\Vert_{L^1(X)}\le C\, \sqrt{t}\, \int_X|\nabla u|\, d\mu.
\]
Hence, if $u\in BV(X)$, then
\[
\Vert P_t u-u\Vert_{L^1(X)}\le C\, \sqrt{t}\, \Vert Du\Vert(X).
\]
\end{lemma}

\begin{proof}
To see the first part of the claim, we note that for each $x\in X$ and $s>0$, $\tfrac{\partial}{\partial s}P_su(x)$ exists,
and so by the fundamental theorem of calculus, for $0<\tau<t$ and $x\in X$,
\[
P_tu(x)-P_\tau u(x)=\int_\tau^t \frac{\partial}{\partial s}P_su(x)\, ds.
\]
Thus for each compactly supported function $\varphi\in \mathcal{F} \cap L^\infty(X)$, 
by the facts that $P_t u$ satisfies the heat equation and that $P_s$ 
is a symmetric operator for each $s>0$,
%commutes with the infinitesimal generator (Laplacian),
\begin{align*}
\bigg\lvert\int_X\varphi(x)[P_tu(x)-P_\tau u(x)]\, d\mu(x)\bigg\rvert
 &=\bigg\lvert-\int_X\int_\tau^t\varphi(x)\frac{\partial}{\partial s}P_su(x)\, ds\, d\mu(x)\bigg\rvert\\
 &=\bigg\lvert\int_\tau^t \int_X d\Gamma(\varphi,P_su)(x)\, ds\bigg\rvert\\
 &=\bigg\lvert\int_\tau^t \int_X d\Gamma(P_s \varphi,u)(x)\, ds\bigg\rvert\\
 %\bigg\lvert\int_X\int_\tau^t\Gamma(P_s\varphi, u)(x)\, ds\, d\mu(x)\bigg\rvert\\
 &\le \int_\tau^t\int_X|\nabla P_s\varphi|\, |\nabla u|\, d\mu\, ds\\
 %\int_X\int_\tau^t\sqrt{\Gamma(P_s\varphi,P_s\varphi)(x)}\, \sqrt{\Gamma(u,u)(x)}\, \, ds\, d\mu(x)\\
 &\le \int_\tau^t  \Vert|\nabla P_s\varphi|\Vert_{L^\infty(X)}\,\int_X|\nabla u|\, d\mu\, ds.
 %\Vert\sqrt{\Gamma(P_s\varphi,P_s\varphi)}\Vert_{L^\infty(X)}\int_X\int_\tau^t \sqrt{\Gamma(u,u)(x)}\, ds\, d\mu(x).
\end{align*}
An application of~\eqref{eq:weak-BE} gives
\begin{align*}
\bigg\lvert\int_X\varphi(x)[P_tu(x)-P_\tau u(x)]\, d\mu(x)\bigg\rvert
&\le C\Vert\varphi\Vert_{L^\infty(X)}\int_\tau^t  \frac{1}{\sqrt s} ds\, \int_X |\nabla u|\,  d\mu\\
%\sqrt{\Gamma(u,u)(x)}\, ds\, d\mu(x)\\
&\le C\, \sqrt t\, \Vert\varphi\Vert_{L^\infty(X)}\, \int_X|\nabla u|\, d\mu.
%\int_X\sqrt{\Gamma(u,u)(x)}\, d\mu(x).
\end{align*}
As the above holds for all compactly supported $\varphi\in \mathcal{F}\cap L^\infty(X)$, %and  as
%compactly supported functions in $\mathcal{F}\cap L^\infty(X)$ are dense in $L^\infty(X)$ 
%(note that $X$ is separable as closed and bounded subsets of $X$ are compact because $X$ is complete and supports a doubling measure, see ~\cite{HKST}) 
%\note{(Do we have any specific reference?)} {\color{red} (Heinonen's lecture notes or ~\cite{HKST} would do.)},
%{\bf (and perhaps we also need to assume that
%this class is also dense in $L^\infty(X)$, which is true if compactly supported Lipschitz functions are dense in $L^\infty(X)$,
%for example for separable $X$)}, 
we obtain
\[
\Vert P_t u-P_\tau u\Vert_{L^1(X)}\le C\,\sqrt{t}\, \int_X|\nabla u|\, d\mu.
\]
Now by the fact that by the fact that $\{P_t\}_{t>0}$ has an
extension as a contraction semigroup to $L^1(X)$ such that $P_\tau u\to u$ as $\tau\to 0^+$ in $L^1(X)$ (see~\cite[Section 2.2]{ABCRST1}), we have 
\[
\Vert P_tu-u\Vert_{L^1(X)}\le C\, \sqrt{t}\, \int_X|\nabla u|\, d\mu.
\] % ({\bf here we need to go from $L^2$-convergence to $L^1$-convergence. Needs to be checked}).

Finally, if $u\in BV(X)$, then we can find a sequence $u_k\in \mathcal{F}\cap W^{1,1}(X)$ 
such that
$u_k\to u$ in $L^1(X)$ and $\lim_{k\to \infty}\int_X|\nabla u_k|\, d\mu=\Vert Du\Vert(X)$. By the contraction property of $P_t$ on $L^1(X)$, we have 
%As $u\in L^1(X)$ and for $t>0$ we have that the heat kernel $p_t(x,y)$ is bounded in $X\times X$,
%%{\bf (Gaussian upper bounds will guarantee this; what are the weaker assumptions?)}, 
%it follows that
%$P_tu\in L^1(X)$ is well-defined with $\Vert P_tu\Vert_{L^1(X)}\le C_t\Vert u\Vert_{L^1(X)}$, see~\cite[Section~1.5]{FOT}. Now,
\begin{align*}
\Vert P_tu-u\Vert_{L^1(X)}
 &\le \Vert P_t(u-u_k)\Vert_{L^1(X)}+\Vert P_t u_k-u_k\Vert_{L^1(X)}+\Vert u_k-u\Vert_{L^1(X)}\\
 &\le C\Vert u-u_k\Vert_{L^1(X)}+C\, \sqrt{t}\, \int_X|\nabla u_k|\, d\mu+\Vert u-u_k\Vert_{L^1(X)}.%\\
% &\to C\, \sqrt{t}\, \lim_{k\to\infty} \int_X|\nabla u_k|\, d\mu=C\, \sqrt{t}\, \Vert Du\Vert(X)\text{ as }k\to\infty.
\end{align*}
Letting $k\to \infty$ concludes the proof.
\end{proof}
 
%Recall the definition of $\mathbf{B}^{1,1/2}(X)$ from \cite{ABCRST1}: $f$ is in this class if 
%\[
%\Vert f\Vert_{\mathbf B^{1,1/2}(X)}:=\Vert f\Vert_{L^1(X)}+\sup_{t>0}\frac{1}{\sqrt{t}}\, \int_X P_t(|f-f(x)|)(x)\, d\mu(x)<\infty.
%\]
%Equivalently,
%\[
%\Vert f\Vert_{\mathbf B^{1,1/2}(X)}=\Vert f\Vert_{L^1(X)}+\sup_{t>0}\frac{1}{\sqrt{t}}\, \int_X\int_Xp_t(x,y)|f(y)-f(x)|\, d\mu(y)\, d\mu(x)<\infty.
%\]
%\note{recall earlier the definition of $\mathbf B^{p,\alpha}(X)$}

Note from the results of~\cite[Theorem~4.1]{MMS} that if the measure $\mu$ is doubling and supports a $1$-Poincar\'e 
inequality, then a measurable set $E\subset X$ is in the BV class if 
\[
\liminf_{t\to 0^+} \frac{1}{\sqrt{t}} \int_{E^{\sqrt{t}}\setminus E} P_t\mathbf{1}_E\, d\mu<\infty.
\]
Here $E^\varepsilon=\bigcup_{x\in E}B(x,\varepsilon)$. Note that  by the symmetry and conservativeness of the operator $P_t$,
\begin{align*}
\int_X|P_t\mathbf{1}_E-\mathbf{1}_E|\, d\mu &=\int_E (1-P_t\mathbf{1}_E)\, d\mu+\int_{X\setminus E}P_t\mathbf{1}_E\, d\mu\\
  &=\int_X\mathbf{1}_E(1-P_t\mathbf{1}_E)\, d\mu+\int_{X\setminus E}P_t\mathbf{1}_E\, d\mu\\
 % &=\int_E P_t\mathbf{1}_{X\setminus E}\, d\mu+\int_{X\setminus E}P_t\mathbf{1}_E\, d\mu\\
  &=\int_X (P_t\mathbf{1}_E)\, \mathbf{1}_{X\setminus E}\, d\mu+\int_{X\setminus E}P_t\mathbf{1}_E\, d\mu
  =2\int_{X\setminus E}P_t\mathbf{1}_E\, d\mu.
\end{align*}
Therefore,
\[
\int_{E^{\sqrt{t}}\setminus E} P_t\mathbf{1}_E\, d\mu\le \int_{X\setminus E}P_t\mathbf{1}_E\, d\mu
=\frac{1}{2}\Vert P_t\mathbf{1}_E-\mathbf{1}_E\Vert_{L^1(X)}.
\]
%The above equality follows from an argument originally due M. Ledoux \cite{Ledoux}. 
%Indeed, by symmetry and conservativeness of the semigroup we have
%\begin{align*}
%\Vert P_t \mathbf{1}_E -\mathbf 1_E  \Vert_{L^1(X,\mu)} =& \int_E (1-P_t \mathbf 1_E ) d\mu + \int_{X\setminus E} P_t(\mathbf 1_E ) d\mu\\
% =& \int_E (P_t \mathbf 1_{X\setminus E} ) d\mu + \int_{X\setminus E} P_t(\mathbf 1_E ) d\mu
% = 2 \int_E (P_t \mathbf1_{X\setminus E}) d\mu.
%\end{align*}
Thus if $\mu$ is doubling and supports a $1$-Poincar\'e inequality, and in addition
\[
\sup_{t>0}\frac{1}{\sqrt{t}}\, \Vert P_t\mathbf{1}_E-\mathbf{1}_E\Vert_{L^1(X)}<\infty,
\]
then $E$ is of finite perimeter. 
In our framework,  those results coming  from \cite{MMS} can not be used, since we do not assume the $1$-Poincar\'e inequality. Instead we prove the following theorem, which is the main result of the section. 

% for which we do not assume the validity of 
%$1$-Poincar\'e inequality, but the weaker version of~\eqref{eq:weak-BE}. {\color{red} (Did we assume 
%$1$-Poincar\'e inequality elsewhere in this section? If not, it might be confusing to have the above line here. Nages)} {\color{Emerald} We have never used $1$-Poincar\'e inequality. This is in contrast with previous results.}

\begin{theorem}\label{thm:W=BV}
If the weak Bakry-\'Emery condition~\eqref{eq:weak-BE} holds, then $\mathbf{B}^{1,1/2}(X)=BV(X)$ with comparable 
seminorms.  Moreover, there exist constants $c,C>0$ such that for every $u \in BV(X)$,
\[
c \limsup_{s  \to 0 }  s^{-1/2}  \int_X P_s (|u-u(y)|)(y) d\mu(y)  
   \le \Vert Du\Vert(X) \le C\liminf_{s \to 0}  s^{-1/2}  \int_X P_s (|u-u(y)|)(y) d\mu(y) .
\]

\end{theorem}

\begin{proof}
First we assume that $u\in BV(X)$. We can assume $u \ge 0$ a.e. We know that for almost every $t \ge 0$ the set $E_t$ is of finite perimeter,
where
\[
E_t=\{x\in X\, :\, u(x)>t\},
\]
and by the co-area formula for BV functions (see Theorem ~\ref{lem:Co-area}),
\[
\Vert Du\Vert(X)=\int_{0}^{+\infty} \Vert D\mathbf{1}_{E_t}\Vert(X)\, dt.
\]
For such $t$, by Lemma~\ref{lem:L1-norm-control} we know that
\[
\sup_{s>0}\frac{1}{\sqrt{s}}\int_X|P_s\mathbf{1}_{E_t}(x)-\mathbf{1}_{E_t}(x)|\, d\mu (x) \le C\, \Vert D\mathbf{1}_{E_t}\Vert(X).
\]
Now, setting $A=\{(x,y)\in X\times X\, :\, u(x)<u(y)\}$, we have for $s>0$,
\begin{align*}
   \int_{X}\int_X & p_s(x,y)  |u(x)-u(y)|\, d\mu(x)d\mu(y) \\
 &\qquad \qquad \qquad =2\int_A\, p_s(x,y)|u(x)-u(y)|\, d\mu(x)d\mu(y)\\
 &\qquad \qquad \qquad =2\int_A\, \int_{u(x)}^{u(y)}\, p_s(x,y)\, dt\, d\mu(x)d\mu(y)\\
 &\qquad \qquad \qquad =2\int_X\int_X \int_0^{+\infty} \mathbf{1}_{[u(x),u(y))}(t)\, \mathbf{1}_A(x,y)\, p_s(x,y)\, dt\, d\mu(x)d\mu(y) \\
 &\qquad \qquad \qquad =2\int_0^{+\infty} \int_X\int_X \mathbf{1}_{E_t}(y)[1-\mathbf{1}_{E_t}(x)]\, p_s(x,y)\, d\mu(x)d\mu(y)\, dt\\
 &\qquad \qquad \qquad =2\int_0^{+\infty} \int_X P_s\mathbf{1}_{E_t}(x)[1-\mathbf{1}_{E_t}(x)]\, d\mu(x)\, dt\\
 &\qquad \qquad \qquad =2\int_0^{+\infty} \int_{X\setminus E_t}P_s\mathbf{1}_{E_t}(x)\, d\mu(x)\, dt.
\end{align*}
Observe that
\[
\int_{X\setminus E_t}P_s\mathbf{1}_{E_t}(x)\, d\mu(x)=\int_{X\setminus E_t}|P_s\mathbf{1}_{E_t}(x)-\mathbf{1}_{E_t}(x)|\, d\mu(x) \le \int_X|P_s\mathbf{1}_{E_t}(x)-\mathbf{1}_{E_t}(x)|\, d\mu(x).
\]
Therefore we obtain
\[
\int_X\int_X p_s(x,y) \ |u(x)-u(y)|\, d\mu(x) d\mu(y)\le 2\int_0^{+\infty} \Vert P_s\mathbf{1}_{E_t}-\mathbf{1}_{E_t}\Vert_{L^1(X)}\, dt.
\]
An application of Lemma~\ref{lem:L1-norm-control} now gives
\[
\int_X\int_X p_s(x,y) \ |u(x)-u(y)|\, d\mu(x)d\mu(y)\le C\sqrt{s}\, \int_0^{+\infty} \Vert D\mathbf{1}_{E_t}\Vert(X)\, dt,
\]
whence with the help of the co-area formula we obtain
\[
\Vert u\Vert_{1,1/2}\le C\, \Vert Du\Vert(X),
\]
that is, $u\in \mathbf{B}^{1,1/2}(X)$. 
Thus $BV(X)\subset \mathbf{B}^{1,1/2}(X)$ boundedly.
%{\color{purple} I think we can get the preceding direction of the proof much more cheaply - I have some notes. Luke} {\color{red}{I agree that we can probably prove $BV(X) \subset  B^{1,1/2}(X)$ under much weaker assumption than weak BE (and it would be interesting to write it), but we also need the continuous inclusion, i.e $\Vert u\Vert_{1,1/2}\le C\, \Vert Du\Vert(X)$ in order to apply the Sobolev and isoperimetric inequalities of Part 1. Can we prove the continuous inclusion  with something weaker than weak BE ? Fabrice }}

\

Now we show that $\mathbf{B}^{1,1/2}(X)\subset BV(X)$. This inclusion holds
even when $\mathcal{E}$ does not support a Bakry-\'Emery curvature
condition; only a $2$-Poincar\'e inequality and the doubling condition on
$\mu$ are needed.

Set $\Delta_\varepsilon=\{(x,y)\in X\, :\, d(x,y)<\varepsilon\}$ for some 
$\varepsilon>0$. Suppose that $u\in \mathbf{B}^{1,1/2}(X)$.  By~\eqref{eq:heat-Gauss2}, we have a Gaussian lower bound for the heat kernel:
\[
 p_t(x,y)\ge \frac{e^{-c\,d(x,y)^2/t}}{C \mu(B(x,\sqrt{t}))}.
\]
%Let $C_0=\Vert u\Vert_{1,1/2}$. 
Therefore  for any $t>0$ we get
\begin{align*}
\frac{1}{\sqrt{t}}\int_X\int_X p_t(x,y) \ |u(x)-u(y)|\, d\mu(x)d\mu(y)&\ge
 \frac{1}{\sqrt{t}} \int_X\int_X\frac{e^{-c\,d(x,y)^2/t}}{C\mu(B(x,\sqrt{t}))}|u(y)-u(x)|\, d\mu(y)\, d\mu(x)\\
%&\ge \iint_{\Delta_\varepsilon} \frac{e^{-c\,d(x,y)^2/t}}{C\mu(B(x,\sqrt{t}))}|u(y)-u(x)|\, d\mu(x)d\mu(y) \\
& \ge  \frac{1}{C\sqrt{t}}\, 
 \iint_{\Delta_{\sqrt{t}} }\frac{|u(y)-u(x)|}{\mu(B(x,\sqrt{t}))}\, d\mu(x)d\mu(y).
\end{align*}
%{\color{blue} In this part of the proof, we do not need the Bakry-\'Emery curvature condition; a Gaussian lower bound
%for the heat kernel would suffice. Such a lower bound is implied by the weaker condition that is $2$-Poincar\'e inequality.}
%Then there is some $C\ge 0$ such that for each $t>0$,
%\[
%\int_X\int_Xp_t(x,y)|u(y)-u(x)|\, d\mu(y)\, d\mu(x)\le C\sqrt{t}.
%\]
%
%\begin{align*}
%C_0\, \sqrt{t}&\ge \int_X\int_X\frac{e^{-c\,d(x,y)^2/t}}{C\mu(B(x,\sqrt{t}))}|u(y)-u(x)|\, d\mu(y)\, d\mu(x)\\
%&\ge \iint_{\Delta_\varepsilon} \frac{e^{-c\,d(x,y)^2/t}}{C\mu(B(x,\sqrt{t}))}|u(y)-u(x)|\, d\mu(x)d\mu(y) \\
%&\ge \frac{e^{-c\varepsilon^2/t}}{C}\, 
% \iint_{\Delta_\varepsilon}\frac{|u(y)-u(x)|}{\mu(B(x,\sqrt{t}))}\, d\mu(x)d\mu(y).
%\end{align*}
%With the choice of $\varepsilon=\sqrt{t}$, and dividing both sides by $\sqrt{t}$ we now get
%\[
%\frac{1}{\sqrt{t}} \int_X\int_X p_t(x,y) \ |u(x)-u(y)|\, d\mu(x)d\mu(y) \ge  \frac{1}{C\sqrt{t}}\, 
% \iint_{\Delta_{\sqrt{t}} }\frac{|u(y)-u(x)|}{\mu(B(x,\sqrt{t}))}\, d\mu(x)d\mu(y).
%\]
%\[
%C_0\, \varepsilon
%\ge \frac{1}{C}\iint_{\Delta_\varepsilon}\frac{|u(y)-u(x)|}{\mu(B(x,\varepsilon))}\, d\mu(x)d\mu(y).
%\]
It follows that
\begin{equation}\label{eq:Ledoux1}
\liminf_{\varepsilon\to 0^+}\frac{1}{\varepsilon}
 \iint_{\Delta_\varepsilon}\frac{|u(y)-u(x)|}{\mu(B(x,\varepsilon))}\, d\mu(x)d\mu(y) <\infty.
\end{equation}
Now an argument as in the second half of the proof of~\cite[Theorem~3.1]{MMS} tells us that 
$u\in BV(X)$. We point out here that although Theorem~3.1 in \cite{MMS} assumes that $X$ supports a $1$-Poincar\'e 
inequality, the second part of the proof there does not need this assumption. In fact,  the argument using
discrete convolution there is valid also in our setting. It is this second part of the proof that we referred to above.
We then obtain
\begin{equation*}
\Vert Du\Vert(X)\le \liminf_{\varepsilon\to 0^+}\frac{1}{\varepsilon}
 \iint_{\Delta_\varepsilon}\frac{|u(y)-u(x)|}{\mu(B(x,\varepsilon))}\, d\mu(x)d\mu(y) 
 \le 
 \Vert u\Vert_{1,1/2}.\qedhere
\end{equation*}
\end{proof}

\begin{remark}\label{chaining metric}
As a byproduct of this proof,  we also obtain that there exists a constant $C>0$ such that for every $u \in BV(X)$,
\begin{align*}
\sup_{\varepsilon >0} \frac{1}{\varepsilon}
 \iint_{\Delta_\varepsilon}\frac{|u(y)-u(x)|}{\mu(B(x,\varepsilon))}\, d\mu(x)d\mu(y) 
 \le  \liminf_{\varepsilon\to 0^+}\frac{C}{\varepsilon}
 \iint_{\Delta_\varepsilon}\frac{|u(y)-u(x)|}{\mu(B(x,\varepsilon))}\, d\mu(x)d\mu(y)
\end{align*}
because both sides are comparable to $\| Du\|(X)$. Indeed, the fact that $\| Du\|(X)$ is dominated by the right hand side is directly from Theorem \ref{thm:W=BV}, which, together with Proposition \ref{prop:B=B} (the metric characterization of Besov spaces), implies that the left hand side can be bounded by $\| Du\|(X)$. This property of the metric measure space $(X,d_\mathcal{E},\mu)$ can be viewed as 
an interesting consequence of the weak Bakry-\'Emery estimate.

\end{remark}

\begin{remark}\label{constancy 1}
Another application of Proposition~\ref{prop:B=B} is the following. It is in general not true that
if $\Vert Du\Vert(X)=0$ then $u$ is constant almost everywhere in $X$, even if $X$ is connected. 
Should $X$ support a $1$-Poincar\'e inequality, it follows immediately that if $\Vert Du\Vert(X)=0$ then  $u$ is constant.
 We can use the above proposition to show that even if we do not have $1$-Poincar\'e 
inequality, if $X$ supports the Bakry-\'Emery curvature condition~\eqref{eq:weak-BE}, then
\[
\sup_{t>0} \int_X\int_{B(x,t)}\frac{|u(x)-u(y)|}{t\mu(B(x,t))}\, d\mu(y)\, d\mu(x)
 \simeq \Vert Du\Vert(X),
\]
and hence if $\Vert Du\Vert(X)=0$ then $u$ is constant.
\end{remark}

\subsection{Sets of finite perimeter}

%We say that $X$ supports a $1$-Poincar\'e inequality if there are constants $C,\lambda>0$ such that 
%whenever $B$ is a ball in $X$ and $u\in W^{1,1}(X)$, we have
%\[
%\int_B|u-u_B|\, d\mu\le C\, \mathrm{rad}(B)\, \int_{\lambda B}|\nabla u|\, d\mu.
%\]
We  introduce some notions from the paper of  Ambrosio~\cite{Ambrosio2002}.
Given $A\subset X$ we set
\[
\mathcal{H}(A):=\lim_{\varepsilon\to 0^+} \inf\bigg\lbrace\sum_i\frac{\mu(B_i)}{\text{rad}(B_i)}\, 
:\, A\subset\bigcup_iB_i, \text{ and }\forall i, \ \text{rad}(B_i)<\varepsilon\bigg\rbrace.
\]
It is known, see~\cite[Proposition 6.3]{KKST}, even without the assumption that $X$ supports a  $2$-Poincar\'e inequality, that if
$\mathcal{H}(\partial E)<\infty$, then $E$ is of finite perimeter. 

Now let $E\subset X$ be a set of finite perimeter and define  the measure-theoretic boundary  by
\[
\partial_mE=\left\{x\in X:\limsup_{r\to 0^+}\frac{\mu(B(x,r)\cap E)}{\mu(B(x,r))}>0,\,
\limsup_{r\to 0^+}\frac{\mu(B(x,r)\setminus E)}{\mu(B(x,r))}>0\right\}.
\]
For $\alpha\in (0,1/2)$, define also 
\[
\partial_\alpha E=\left\{x\in X:\liminf_{r\to 0^+} \min\left\{\frac{\mu(B(x,r)\cap E)}{\mu(B(x,r))},\,
\frac{\mu(B(x,r)\setminus E)}{\mu(B(x,r))}\right\}> \alpha\right\}.
\]

%In the event that $X$ supports a $1$-Poincar\'e inequality, from the results of
%Ambrosio, Miranda and Pallara \Rd in \cite{AMP}, \Bk we know that $P(E,X)\simeq \mathcal{H}(\partial_mE)$
%whenever $E$ is of finite perimeter. \note{This comment overlaps with the last paragraph of this section. Li} Here 

If $X$ supports a $1$-Poincar\'e inequality then, by the results of~\cite[Theorems~5.3, 5.4]{Ambrosio2002}, %~\cite[Theorem 4.4]{AMP},
there is  $\gamma>0$ such that
$\mathcal{H}(\partial_m E\setminus \partial_\gamma E)=0$.  Moreover, $\mathcal{H}(\partial_mE)<\infty$, and $P(E,\cdot)\ll\mathcal{H}|_{\partial_mE}$ with Radon-Nikodym derivative bounded below by some $\delta>0$. Both $\gamma$ and $\delta$ depend solely on the doubling and the $1$-Poincar\'e constants.
% and moreover that if $E$ is of finite perimeter then $\mathcal{H}(\partial_m E)\simeq P(E,X)$.

%\note{(The original remark after the proof is moved here. In this case, I add the definition of the set $\partial_\gamma E$. However, in the proof of the proposition below, we also define $\partial_\alpha E$ which would be the same. Also, in the definitions of the boundary sets, I change $\partial E$ to $X$ following \cite{AMP}. We need to change them back in case it is not accurate.)}
%{\color{blue} I changed $\gamma$ to $\alpha$ in the above $\partial_\alpha E$ to be consistent with 
%the notation for $\partial_\alpha^{r_0}E$. Nages}

We are not assuming $X$ supports a $1$-Poincar\'e inequality, but only that $\mu$ is doubling and  $X$ supports a $2$-Poincar\'e inequality, in which case we have the following bound. %In this case, we have the double Gaussian bounds~\eqref{eq:heat-Gauss} for the heat kernel $p_t(x,y)$. 
%n this setting we instead consider for $r_0>0$ and $0<\alpha\le 1/2$ the quantity
% set 
%\[
%\partial_\alpha^{r_0}E=\left\{x\in\partial E: \frac{\mu(B(x,r)\cap E)}{\mu(B(x,r))}>\alpha, \quad %\qquad \qquad
%\frac{\mu(B(x,r)\setminus E)}{\mu(B(x,r))}>\alpha,\,\forall \,0<r\le r_0.\right\}
%\]
% we set %$\partial_\alpha^{r_0}E$ to be the collection of all
%points $x\in\partial E$ such that for all $0<r\le r_0$,
%\[
%\partial_\alpha^{r_0}E=\left\{x\in X: \min\left\{\frac{\mu(B(x,r)\cap E)}{\mu(B(x,r))}, \, %\qquad \qquad
%\frac{\mu(B(x,r)\setminus E)}{\mu(B(x,r))}\right\}>\alpha \text{ for all } 0<r\le r_0\right\}.
%\]
%Observe that $\partial_mE=\bigcup_{0<\alpha<1}\bigcup_{0<r_0<1}\partial_\alpha^{r_0}(E)$ and the union can be made countable by taking $\alpha$ and $r_0$ to be rational numbers.

\begin{proposition}\label{lem:Hausdorff-perimeter}
Suppose that
$E\subset X$ with $\Vert \mathbf{1}_E\Vert_{\mathbf{B}^{1,1/2}(X)}<\infty$. Then for all
%$r_0>0$ and
 $0<\alpha<1$,
%\[
%\mathcal{H}(\partial_\alpha^{r_0}E)\le \frac{C}{\alpha}\, P(E,X).
%\]
%Consequently, 
we have $\mathcal{H}(\partial_\alpha E)\le \frac{C}{\alpha}\, P(E,X)$.% and $\mathcal{H}\vert_{\partial_\alpha E}$ is a $\sigma$-finite measure. 
\end{proposition}

\begin{remark} %\bf 
According to~\cite{Lahti}, if $(X,d_{\mathcal{E}},\mu)$ is doubling and supports a $1$-Poincar\'e inequality,
then there is $\alpha$, depending solely on the doubling and the 
Poincar\'e constants, such that finiteness of $\mathcal{H}(\partial_\alpha E)$ implies that $P(E,X)$ is finite and $\mathcal{H}(\partial_m E\setminus\partial_\alpha E)=0$. \end{remark}

\begin{proof}
For $r_0>0$ and $0<\alpha\le 1/2$ let
\begin{align*}
\partial_\alpha^{r_0}E
	&=\left\{x\in X: \min\left\{\frac{\mu(B(x,r)\cap E)}{\mu(B(x,r))}, \, \frac{\mu(B(x,r)\setminus E)}{\mu(B(x,r))}\right\}>\alpha \text{ for all } 0<r\le r_0\right\} \\
	&= \{\partial_\alpha^{r_0}E=\{x\in X\, :\, \Phi_{E,r_0}(x)>\alpha\}\cap\{x\in X\, :\, \Phi_{X\setminus E,r_0}(x)>\alpha\}, \text{ where}\\
	\Phi_{E,r_0}&:=\inf_{r\in\mathbb{Q}\cap(0,r_0]} \frac{\mu(B(x,r)\cap E)}{\mu(B(x,r))}.
\end{align*}
Since $\mu$ is Borel regular we know $x\mapsto \mu(B(x,r)\cap E)$ is lower semicontinuous, so $\frac{\mu(B(x,r)\cap E)}{\mu(B(x,r))}$ and $\Phi_{E,r_0}$ are  Borel functions.  It follows that the sets $\partial_\alpha^{r_0}E$ are Borel, and we conclude by writing $\partial_\alpha E= \bigcup_{(0,1)\cap\mathbb{Q}}\partial_\alpha^{r_0}E$ both that $\partial_\alpha E$ is Borel and, by continuity of measure, that we need only prove $\mathcal{H}(\partial_\alpha^{r_0}E)\le \frac{C}{\alpha}\, P(E,X)$ for each $r_0$.

Using Theorem~\ref{thm:W=BV}, if $\mathbf{1}_E\in BV(X)$ then
\begin{equation*}
  \sup_{t>0}\frac{1}{\sqrt{t}}\int_X\int_Xp_t(x,y)|\mathbf{1}_E(x)-\mathbf{1}_E(y)|\, d\mu(x)\, d\mu(y)
  \leq C \|D\mathbf{1}_E\|(X) = C P(E,X)
  \end{equation*}
Fix $t<(r_0/3)^2\leq 1/36$ and let $\{B_i\}_i$ be a maximally separated $\sqrt{t}$-covering of  $\partial_\alpha^{r_0}E$, so the balls $5B_i$ have bounded overlap (see Section \ref{sec:D-P}).  Observe that for $x,y\in B_i$ the  Gaussian lower bound for $p_t(x,y)$ in~\eqref{eq:heat-Gauss} becomes $p_t(x,y)\geq C\mu(B_i)^{-1}$. (In this calculation $C$ denotes various constants that can change even within an expression, but depend only on the doubling and Poincar\'e constants of the space.)  Thus
\begin{align*}
\sqrt{t}\, P(E,X)&\ge C\sum_i\int_{B_i\cap E}\int_{B_i\setminus E}
p_t(x,y)\, d\mu(x)\, d\mu(y)\\
&\ge C\sum_i \frac1{\mu(B_i)} \int_{B_i\cap E}\int_{B_i\setminus E} 
  d\mu(x)\, d\mu(y)\\
&\ge C \sum_i\frac{\mu(B_i\cap E)\,\mu(B_i\setminus E)\, \mu(B_i)}{\mu(B_i)^2}
\geq C \alpha \sum \mu(B_i)
\end{align*}
where the last inequality uses that at least one of $\mu(B_i\cap E)$ and $\mu(B_i\setminus E)$
is larger than $\mu(B_i)/2$ and the other is bounded below by $\alpha\mu(B_i)$ on $\partial _\alpha^{r_0}E$.  However each $B_i$ has radius $\sqrt{t}$, so
\[
P(E,X)\ge C \alpha \sum_i\frac{\mu(B_i)}{\mathrm{rad}(B_i)}
\]
and thus $P(E,X)\geq C\alpha\mathcal{H}(\partial _\alpha^{r_0}E)$, completing the proof.
\end{proof}

%; however, it is not known that
%if $E$ is a measurable set with $\mathcal{H}(\partial_m E)$ finite, then $E$ is of
%finite perimeter (there are some additional geometric conditions on $X$ that would
%give this, see for example {\bf Korte, Lahti, Shanmugalingam}, but this additional
%geometric conditions are not always satisfied). Given this gap in our knowledge, 
%the above lemma gives a more measure-theoretic (zeroth-order calculus)
%criterion for $E$ to have a finite
%perimeter.

Proposition \ref{lem:Hausdorff-perimeter}  gives us a way to control, from above, the $\mathcal{H}$-measure of 
$\partial_\alpha E$ for a set $E$ of finite perimeter.  This should be contrasted with the following lower bound on the  co-dimension~$1$ Minkowski measure of $\partial E$.  For a set $A\subset X$,
the co-dimension $1$-Minkowski measure of $A$ is defined to be
\[
\mathcal{M}_{-1}(A):=\liminf_{\varepsilon\to 0^+} \frac{\mu(A_\varepsilon)}{\varepsilon},
\]
where $A_\varepsilon=\bigcup_{x\in A}B(x,\varepsilon)$.
 
\begin{proposition}
%Assuming the weak Bakry-\'Emery condition~\eqref{eq:weak-BE}  
{\bf} We have for a set $E$ of finite perimeter that
\begin{equation*}
P(E,X)  \leq  \mathcal{M}_{-1}(\partial E).
\end{equation*}
\end{proposition}

\begin{proof}
We can assume $\mathcal{M}_{-1}(\partial E)<\infty$. For each $\varepsilon>0$, consider
\[
u_\varepsilon(x)=\min\{1,\varepsilon^{-1}\text{dist}_{d_{\mathcal{E}}}(x,X\setminus A_\varepsilon)\},
\]
where $\text{dist}_{d_{\mathcal{E}}}(x,A)=\inf\{d_{\mathcal{E}}(x,y)\, :\, y\in A\}$ and $A=\partial E$. Lemma~\ref{lem:diff-of-Lip} together with the fact that $u_\varepsilon\to\chi_E$ as $\varepsilon\to 0^+$
give the desired result.
%Observe that $|\mathbf{1}_E(x)-\mathbf{1}_E(y)|$ is bounded by $1$ and is zero if $y\in B(x,t)$ and $x$ is not in the $2t$ neighborhood $(\partial E)_{2t}$ of the boundary.  Using the doubling property of $\mu$ we immediately deduce
%\[
%\liminf_{t\to 0^+}\int_X\int_{B(x,t)}\frac{|\mathbf{1}_E(x)-\mathbf{1}_E(y)|}{ t\mu(B(x,t))}\, d\mu(y)\, d\mu(x)
%\le C \mathcal{M}_{-1}(\partial E).
%\]  
%
%Now recall from Remark~\ref{chaining metric} that under the weak Bakry-\'Emery assumption the above limit inferior is comparable to the supremum, which is the $B^1_{1,\infty}$ norm defined in~\eqref{eq:BesovMetric-q=infty}, because both are comparable to  the perimeter measure $P(E,X)=\|D\mathbf{1}_E(X)\|$.
\end{proof}

\subsection{Under the quasi Bakry-\'Emery condition, $\mathbf{B}^{p,1/2}(X)=W^{1,p}(X)$ for $p >1$}

In this section we compare the Besov and Sobolev seminorms for $p>1$. The case $p=1$ was studied in detail in  Section 4.2. Our main theorem in this section is the following:

\begin{theorem}\label{eq:Besov-Sobolev}
Suppose that the quasi Bakry-\'Emery condition~\eqref{eq:strong-BE} holds. Then, for every $p>1$, $\mathbf{B}^{p,1/2}(X)=W^{1,p}(X)$ with comparable norms.
\end{theorem}

We will divide the proof of Theorem \ref{eq:Besov-Sobolev} in two parts. In the first part, Theorem~\ref{thm:BesovLB}, we prove that $\mathbf{B}^{p,1/2}(X) \subset W^{1,p}(X)$. As we will see, this inclusion does not require the quasi Bakry-\'Emery condition~\eqref{eq:strong-BE}. In the second part, Theorem~\ref{thm:BesovUB} we will prove the inclusion $ W^{1,p}(X) \subset \mathbf{B}^{p,1/2}(X)$ and, to this end, will use the quasi Bakry-\'Emery condition. Before turning to the proof, we point out the following corollary regarding the Riesz transform.
\begin{corollary}\label{thm:RT}
Suppose that the quasi Bakry-\'Emery condition~\eqref{eq:strong-BE} holds. Let $p>1$. Then for any $f\in \mathbf B^{p,1/2}(X) \cap \mathcal F$,
\[
 \Vert  f\Vert_{p,1/2} \simeq \Vert \sqrt{-L} f \Vert_{L^p(X)}.
\]
Consequently, $\mathbf B^{p,1/2}(X) =\mathcal L_p^{1/2}$, where $\mathcal{L}_p^{1/2}$  is the domain of the 
operator $\sqrt{-L}$ in $L^p(X)$ (see {\rm \cite[Section 4.6]{ABCRST1}} for the definition).
\end{corollary}
\begin{proof}
In view of  Theorem \ref{eq:Besov-Sobolev}, we have that for  any $f\in \mathbf B^{p,1/2}(X) $
\[
 \Vert  f\Vert_{p,1/2} \simeq \Vert |\nabla f |\Vert_{L^p(X)}.
\]
On the other hand, it follows from \cite[Theorem 1.4]{ACDH} that for  any $f\in \mathcal L_p^{1/2}$,
\[
 \Vert  \sqrt{-L} f  \Vert_{L^p(X)} \simeq \Vert |\nabla f |\Vert_{L^p(X)}.
\]
We conclude the proof by combining the above two facts.
%\note{It is not clear to write in this way.}
\end{proof}
%\begin{corollary}\label{thm:Continuity}
%Suppose that the quasi Bakry-\'Emery condition~\eqref{eq:strong-BE} holds. Let $p>1$. Then for any $f\in L^p(X)$ and $t>0$
%\[
% \Vert P_t f\Vert_{p,1/2} \leq \frac{C_p}{t^{1/2}}\Vert  f \Vert_{L^p(X)}.
%\]
%\end{corollary}
%
%\begin{proof}
% This is a consequence of Lemma \ref{lem:LpGradient} and Theorem \ref{thm:BesovUB},
%\end{proof}

\subsubsection{$\mathbf{B}^{p,1/2}(X) \subset W^{1,p}(X)$}

%{\color{red} Here we seem to only have shown that $\mathbf{B}^{p,1/2}(X)\cap\mathcal{F}\subset W^{1,p}(X)$, which is
%not the same thing. On the other hand, I do not see why we need $u\in\mathcal{F}$ in the proof given below? 
%I think we only need to know that locally Lipschitz functions such as $u_\eps$ are in $\mathcal{F}_{loc}$. Nages}
%
%{\color{blue}{I agree, we don't seem to need that $u \in \mathcal{F}$ in the proof below. Fabrice}}

\begin{theorem}\label{thm:BesovLB}
Let $p >1$. There exists a constant $C>0$ such that for every $ u \in \B^{p,1/2}(X)$,
\[
\| |\nabla u| \|_{L^p(X)} \le C \| u \|_{p,1/2}.
\]
\end{theorem}

\begin{proof}
Let $u\in\B^{p,1/2}(X)$. Then from Proposition \ref{prop:B=B}, we see that
%Let $u\in L^p(X)\cap L^2(X)$ such that $\|u\|_{p,1/2}$ is finite. Then, from Theorem \ref{thm:W=BV} we see that 
for each $\eps>0$, 
\[
\frac{1}{\eps^p}\iint_{\Delta_\eps}\frac{|u(x)-u(y)|^p}{\mu(B(x,\eps))}\, d\mu(y)\, d\mu(x)\le \| u \|_{p,1/2}^p<\infty.
\]
Fix $\eps>0$. 
As in the proof of Lemma \ref{absolute continuity}, let $\{B_i^\eps=B(x_i^\eps,\eps)\}_i$ be a maximally separated $\eps$-covering (Definition~\ref{def:maxsepcover} and $\{\pip_i^\eps\}_i$ be a $(C/\eps)$-Lipschitz partition of unity subordinated to this covering. We also set 
%Fix $\eps>0$ and cover $X$ by family of balls $B_i^\eps=B(x_i^\eps,\eps)$ such that $\frac12 B_i^\eps$ are pairwise
%disjoint, and let $\pip_i^\eps$ be a $(C/\eps)$-Lipschitz partition of unity subordinate to this cover: that is, 
%$0\le \pip_i^\eps\le 1$ on $X$, $\sum_i\pip_i^\eps=1$ on $X$, and $\pip_i^\eps=0$ in $X\setminus B_i^\eps$. We then
%set 
\[
u_\eps:=\sum_i u_{B_i^\eps}\, \pip_i^\eps.
\]
 Then $u_\eps$ is locally Lipschitz and hence is in $\mathcal{F}_{\mathrm{loc}}(X)$. %(see  the proof of Lemma \ref{absolute continuity}).
Indeed, for $x,y\in B_j^\eps$ we see that
\begin{align*}
|u_\eps(x)-u_\eps(y)|&\le \sum_{i:2B_i^\eps\cap 2B_j^\eps\ne\emptyset}|u_{B_i^\eps}-u_{B_j^\eps}||\pip_i^\eps(x)-\pip_i^\eps(y)|\\
 &\le \frac{C\, d(x,y)}{\eps} \sum_{i:2B_i^\eps\cap 2B_j^\eps\ne\emptyset}
    \left(\vint_{B_i^\eps}\vint_{B(x,2\eps)}|u(y)-u(x)|^p\, d\mu(y)\, d\mu(x)\right)^{1/p}.
\end{align*}
Therefore, by Lemma~\ref{lem:diff-of-Lip}, we see that
\begin{align*}
|\nabla u_\eps|&\le \frac{C}{\eps}\sum_{i:2B_i^\eps\cap 2B_j^\eps\ne\emptyset}
    \left(\vint_{B_i^\eps}\vint_{B(x,2\eps)}|u(y)-u(x)|^p\, d\mu(y)\, d\mu(x)\right)^{1/p}\\
    &\le C\left(\vint_{2B_j^\eps} \vint_{B(x,2\eps)}\frac{|u(y)-u(x)|^p}{\eps^p}\, d\mu(y)\, d\mu(x)\right)^{1/p},
\end{align*}
and so by the bounded overlap property of the collection $2B_j^\eps$,
\begin{align*}
\int_X|\nabla u_\eps|^p\, d\mu &\le \sum_j \int_{B_j^\eps}|\nabla u_\eps|^p\, d\mu\\
  &\le C\, \sum_j \int_{2B_j^\eps} \vint_{B(x,2\eps)}\frac{|u(y)-u(x)|^p}{\eps^p}\, d\mu(y)\, d\mu(x)\\
  &\le C\, \int_X \vint_{B(x,2\eps)}\frac{|u(y)-u(x)|^p}{\eps^p}\, d\mu(y)\, d\mu(x)\\
  &\le C\, \frac{1}{\eps^p}\int_{\Delta_{2\eps}}\frac{|u(x)-u(y)|^p}{\mu(B(x,\eps))}\, d\mu(y)\, d\mu(x)\le C\,\|u\|_{p,1/2}^p.
\end{align*}
Hence we have 
\begin{equation}\label{eq:sup-gradient}
\sup_{\eps>0} \int_X|\nabla u_\eps|^p\, d\mu \le C\,\|u\|_{p,1/2}^p.
\end{equation}
In a similar manner, we can also show that 
\[
\int_X|u_\eps(x)-u(x)|^p\, d\mu(x)\le C \eps^p \int_{\Delta_{2\eps}}\frac{|u(x)-u(y)|^p}{\eps^p\, \mu(B(x,\eps))}\, d\mu(y)\, d\mu(x)
  \le C\, \eps^p \,\|u\|_{p,1/2}^p,
\]
that is, $u_\eps\to u$ in $L^p(X)$ as $\eps\to 0^+$. 
%So we get that if $u\in\mathcal{F}$, then
%\[
%C\, \| u\|_{p,1/2}\ge \| |\nabla u| \|_{L^p(X)}.
%\]\note{Do we need to add the following to make the proof more self-contained?}

Take a sequence $\eps_n\to 0^+$. From \eqref{eq:sup-gradient} and the reflexivity of $L^p(X)$, there exists a subsequence of $\{\nabla u_{\eps_n}\}_n$ that is weakly convergent in $L^p(X)$. By Mazur's lemma, a sequence of convex combinations of $u_{\eps_n}$ converges in the norm of $W^{1,p}(X)$. Since it converges to $u$ in $L^p(X)$, we conclude that $u\in W^{1,p}(X)$ and hence 
\begin{equation*}
\| |\nabla u| \|_{L^p(X)} \le C \| u \|_{p,1/2}.\qedhere
\end{equation*}
\end{proof}

\subsubsection{$W^{1,p}(X) \subset \mathbf{B}^{p,1/2}(X)$}

We now turn to the proof of the upper bound for the Besov seminorm in terms of the Sobolev seminorm and assume that the quasi Bakry-\'Emery condition~\eqref{eq:strong-BE} holds.

A first important corollary of the quasi Bakry-\'Emery estimate is the following Hamilton's type gradient estimate for the heat kernel. This type of estimate is well-known on Riemannian manifolds with non-negative Ricci curvature (see for instance  \cite{Kotschwar}), but is new in our general framework. 

\begin{theorem}\label{Hamilton estimate}
There exists a constant $C>0$ such that for every $t>0$, $x,y \in X$,
\[
| \nabla_x \ln p_t (x,y) |^2 \le \frac{C}{t} \left( 1 +\frac{d(x,y)^2}{t} \right).
\]
\end{theorem}

\begin{proof}
The proof proceeds in two steps. 

\textbf{Step 1:}  We first collect a gradient bound for the heat kernel. Observe that \eqref{eq:strong-BE} implies a weaker $L^2$ version as follows
\[
 |\nabla P_t u|^2  \le C P_t (|\nabla u|^2), 
\]
and hence the following pointwise  heat kernel gradient bound (see \cite[Lemma 3.3]{ACDH}) holds:
\[
|\nabla_x p_t(x,y)| \le \frac{C}{\sqrt t} \frac{e^{-cd(x,y)^2/t}}{\sqrt{\mu(B(x,\sqrt t)) \mu(B(y,\sqrt t))}}. 
\]
In particular, we note that $|\nabla_x p_t(x,\cdot)| \in L^p(X)$ for every $p \ge 1$.

\

\textbf{Step 2:} In the second step, we prove a reverse log-Sobolev inequality for the heat kernel. 
Let $\tau,\ve>0$ and $x\in X$ be fixed.  We denote $u=p_\tau (x,\cdot)+\ve$. One has, from the chain rule for strictly local forms \cite[Lemma~3.2.5]{FOT},
\begin{align}
P_t (u \ln u)-P_t u \ln P_t u &=\int_0^t \partial_s \left( P_s (P_{t-s} u  \ln P_{t-s} u) \right) ds \notag\\
&=\int_0^t LP_s (P_{t-s} u \ln P_{t-s} u) -P_s (LP_{t-s}u \ln P_{t-s} u) -P_s(LP_{t-s}u)  ds \notag\\
&=\int_0^tP_s (L(P_{t-s} u \ln P_{t-s} u)) -P_s (LP_{t-s}u \ln P_{t-s} u) -P_s(LP_{t-s}u)  ds\notag\\
&=\int_0^tP_s \left[ L(P_{t-s} u \ln P_{t-s} u)) -LP_{t-s}u \ln P_{t-s} u -LP_{t-s}u \right]  ds \notag\\
&=\int_0^t   2P_s\left(\frac{| \nabla P_{t-s}u |^2 }{P_{t-s} u } \right)   ds, \label{eq:revlogsobeqn1}
 \end{align}
 where the above computations may be justified by using the Gaussian heat kernel estimates for the heat kernel and the Gaussian upper bound for the gradient of the heat kernel obtained in Step~1. In particular, we point out that the commutation $ LP_s (P_{t-s} u \ln P_{t-s} u) =P_s (L(P_{t-s} u \ln P_{t-s} u))$ is justified by noting that $P_{t-s} u \ln P_{t-s} u -\ve \ln \ve$ is in the domain of $L$ in $L^2(X,\mu)$.
Here, $L$ is the infinitesimal generator (the Laplacian operator) associated with $\mathcal{E}$. 
 
Using the Cauchy-Schwarz inequality in the form $P_s\left( \frac{f^2}{g}\right) \ge \frac{(P_sf)^2}{P_s g}$ and then the quasi Bakry-\'Emery estimate, we obtain from~\eqref{eq:revlogsobeqn1}
\begin{align*}
P_t (u \ln u)-P_t u \ln P_t u 
 & \ge 2\int_0^t  \frac{(P_s | \nabla P_{t-s} u |)^2}{  P_s(P_{t-s} u) }    ds \\
 &\ge \frac{1}{C} \frac{1}{P_t u} \int_0^t | \nabla P_t u|^2 ds \\
  &\ge \frac{t}{C}   \frac{1}{P_t u}  | \nabla P_t u|^2.
 \end{align*}
% Therefore, by letting $\ve \to 0$
% \[
% P_t ( p_\tau (x,\cdot) \ln (p_\tau (x,\cdot))-P_t p_\tau (x,\cdot) \ln P_t p_\tau (x,\cdot) \ge \frac{t}{C}   \frac{1}{P_t p_\tau (x,\cdot)}  | \nabla P_t p_\tau (x,\cdot)|^2 .
% \]

Coming back to the definition of $u$, noting that $P_t p_\tau(x,\cdot)=p_{t +\tau}(x,\cdot)$  and applying the previous inequality with $t=\tau$, we may set $M_t(x)=\sup_{y \in X} p_t (x,y) $ and bound the $P_t(u\ln u)$ term by $(P_t u)\ln (M_t+\epsilon)$ to deduce
\[
| \nabla_{y} \ln (p_{2t} (x,y)+\ve)|^2\le \frac{C}{t} P_t \left[  \ln \left( \frac{M_t(x)+\ve }{p_{2t} (x,\cdot) +\ve} \right) \right] (y).
\]
By letting $\ve \to 0$ and using the Gaussian heat kernel estimate, one concludes
\[
| \nabla_y \ln p_{2t} (x,y) |^2 \le \frac{C}{t} \left( 1 +\frac{d(x,y)^2}{t} \right).
\]
Our desired inequality follows by rescaling $t$, adjusting the constant $C$ and using the symmetry of $p_t(x,y)$ in $x$ and $y$ in the whole above argument.
\end{proof}

\begin{corollary}\label{cor:ptwiseboundfornablaPtbyPtLp}
Let $p>1$. There exists a constant $C>0$ such that for every $u \in L^p(X)$,
\[
| \nabla P_t u| \le \frac{C}{\sqrt{t}} (P_t |u|^p)^{1/p}.
\]
\end{corollary}

\begin{proof}
Let $p>1$, $q$  be the conjugate exponent and $u \in L^p(X)$. One has from H\"older's inequality
\begin{align*}
| \nabla P_t u| (x) & \le \int_X |\nabla _x p_t (x,y)| |u(y)| d\mu(y) \\
& \le \left( \int_X \frac{|\nabla _x p_t (x,y)|^q}{p_t(x,y)^{q/p}}  d\mu(y)\right)^{1/q} (P_t |u|^p)^{1/p} \\
&\le  \left( \int_X | \nabla_x \ln p_t(x,y)|^q p_t(x,y)  d\mu(y)\right)^{1/q} (P_t |u|^p)^{1/p}.
\end{align*}
The proof follows then from Theorem \ref{Hamilton estimate} and the Gaussian upper bound for the heat kernel.
\end{proof}

Note that by integrating over $X$ the previous proposition immediately yields:
\begin{lemma}\label{lem:LpGradient}
Let $p>1$. There exists a constant $C>0$ such that for every $u \in L^p(X)$
\[
\Vert |\nabla P_tu| \Vert_{L^p(X)}^2\le \frac{C}{t}\Vert u\Vert_{L^p(X)}^2.
\]
\end{lemma}

From this estimate we obtain the following result.
\begin{lemma}\label{lem:HSminusF1}
Let  $p > 1$. There exists a constant $C>0$ such that for every $u \in L^p(X)\cap\mathcal{F}$ with
$|\nabla u|\in L^p(X)$
\[
\| P_t u -u \|_{L^p(X)} \le C \sqrt{t} \| |  \nabla u | \|_{L^p(X)}.
\]
\end{lemma}

\begin{proof}
With the previous lemma in hand, the proof is similar to the one in Lemma \ref{lem:L1-norm-control}, with $\varphi$ in 
$\mathcal{F}\cap L^q(X)$ and compactly supported, where $p^{-1}+q^{-1}=1$.  As compactly supported functions in $\mathcal{F}\cap L^q(X)$ form a dense subclass of $L^q(X)$ we recover
the $L^p$-norm of $P_tu-u$ by taking the supremum over all such $\varphi$ with $\int_X|\varphi|^q\, d\mu\leq 1$.
\end{proof}

\begin{lemma}\label{lem:HSminusF2}
Let $p >1$, then for every $u \in L^p(X)\cap\mathcal{F}$ with
$|\nabla u|\in L^p(X)$
\[
\left( \int_X \int_X | P_tu(x)-u (y) |^p p_t (x,y) d\mu(x) d\mu(y) \right)^{1/p} \le C \sqrt{t}  \| |  \nabla u | \|_{L^p(X)}.
\]
\end{lemma}

\begin{proof}
 Let $u\in L^p(X)$ and $t >0 $ be fixed in the above argument. By an application of Fubini's theorem we have
\[
\left( \int_X \int_X | P_tu(x)-u (y) |^p p_t (x,y) d\mu(x) d\mu(y) \right)^{1/p}=\left( \int_X P_t ( | P_t u(x) -u|^p)(x) d\mu(x)\right)^{1/p}.
\]

The main idea now is to  adapt the proof of~\cite[Theorem~6.2]{BBBC}.
As above, let $q$ be the conjugate of $p$. Let $x \in X$ be fixed. Let $g$ be a function in $L^\infty(X)$  such that $P_t (|g|^q)(x) \le 1$. 

We first note that from the chain rule:
\begin{align*}
& \partial_s \left[ P_s  ((P_{t-s} u) (P_{t-s} g))(x) \right]  \\
=& LP_s  ((P_{t-s} u) (P_{t-s} g))(x) -P_s  ((LP_{t-s} u) (P_{t-s} g))(x) -P_s  ((P_{t-s} u) (LP_{t-s} g))(x) \\
=&P_s  (L(P_{t-s} u) (P_{t-s} g))(x) -P_s  ((LP_{t-s} u) (P_{t-s} g))(x) -P_s  ((P_{t-s} u) (LP_{t-s} g))(x)  \\
=&2P_s (\Gamma(P_{t-s}u,P_{t-s}g)).
\end{align*}

Therefore we have
\begin{align*}
P_t ( ( u-P_t u(x) ) g)(x) & =P_t (ug)(x)-P_t u(x) P_t g (x) \\
 &=\int_0^t \partial_s \left[ P_s  ((P_{t-s} u) (P_{t-s} g))(x) \right] ds \\
 &=2 \int_0^t P_s \left( \Gamma(P_{t-s} u,P_{t-s} g)  \right) (x) ds \\
 &\le 2 \int_0^t P_s \left( | \nabla P_{t-s} u | | \nabla P_{t-s} g| )  \right) (x) ds \\
 &\le 2 \int_0^t P_s \left( | \nabla P_{t-s} u |^p \right)^{1/p}(x) P_s\left(  | \nabla P_{t-s} g|^q  \right)^{1/q} (x) ds.
\end{align*}
Now from the strong  Bakry-\'Emery estimate and H\"older's inequality we have
\[
P_s \left( | \nabla P_{t-s} u |^p \right)^{1/p}(x) \le  CP_s  \left( P_{t-s}( | \nabla  u |^p) \right)^{1/p}(x) =C P_t ( | \nabla  u |^p)^{1/p}(x).
\]
On the other hand, Corollary~\ref{cor:ptwiseboundfornablaPtbyPtLp} gives
\[
 | \nabla P_{t-s} g|^q\le \frac{C}{(t-s)^{q/2}} P_{t-s} (|g|^q).
 \]
 Thus,
 \[
 P_s\left(  | \nabla P_{t-s} g|^q  \right)^{1/q} (x) \le \frac{C}{(t-s)^{1/2}} P_{t} (|g|^q)^{1/q}(x) \le \frac{C}{(t-s)^{1/2}} .
 \]
 One concludes
 \[
 P_t ( ( u-P_t u(x) ) g)(x) \le C \sqrt{t} P_t ( | \nabla  u |^p)^{1/p}(x).
 \]
 Thus by $L^p-L^q$ duality in $(X,p_t(x,y)\mu(dy))$, one concludes
 \[
  P_t ( | u-P_t u(x)| ^p)(x)^{1/p} \le C \sqrt{t} P_t ( | \nabla  u |^p)^{1/p}(x) 
 \]
 and finishes the proof by integration over $X$.
\end{proof}

We are finally in a position to prove the inclusion of the Sobolev space $W^{1,p}(X)$ into the Besov class $\mathbf{B}^{p,1/2}$, which in turn completes the proof of Theorem~\ref{eq:Besov-Sobolev}, which is the main result of this section.

\begin{theorem}\label{thm:BesovUB}
Let $p>1$. There exists a constant $C>0$ such that for every  $u \in W^{1,p}(X)$,
\[
\| u \|_{p,1/2} \le C \| |  \nabla u | \|_{L^p(X)}.
\]
\end{theorem}

\begin{proof}
We first assume $u \in L^p(X)\cap \mathcal{F}$ with $|\nabla u| \in L^p(X)$.
One has
\begin{align*}
 \lefteqn{\left( \int_X \int_X | u(x)-u(y)|^p p_t (x,y) d\mu(x) d\mu(y) \right)^{1/p} }\quad&  \\
 &\leq  \left( \int_X \int_X | u(x)-P_tu (x) |^p p_t (x,y) d\mu(x) d\mu(y) \right)^{1/p} + \left( \int_X \int_X | P_tu(x)-u(y)|^p p_t (x,y) d\mu(x) d\mu(y) \right)^{1/p}   \\
 &\leq  \| P_t u -u \|_{L^p(X)} +  \left( \int_X \int_X | P_tu(x)-u (y) |^p p_t (x,y) d\mu(x) d\mu(y) \right)^{1/p}\\
 & \leq  2 C \sqrt{t} \| |  \nabla u | \|_{L^p(X)},
\end{align*}
where in the last step we applied Lemma~\ref{lem:HSminusF1} to the first term and Lemma~\ref{lem:HSminusF2} to the second term. Thus
 \[
\| u \|_{p,1/2} \le C \| |  \nabla u | \|_{L^p(X)}.
\]
Now let $u \in W^{1,p}(X)$ and choose an increasing sequence
of functions $\phi_n\in C^\infty([0,\infty))$ such that
$\phi_n\equiv 1$ on $[0,n]$, $\phi_n \equiv 0$ outside $[0,2n]$, and
$|\phi_n'|\le \frac{2}{n}$. Let $x_0 \in X$. If  $h_n(x) = \phi_n(d(x_0,x))$
then $h_nu \in \mathcal{F}$, $h_n\nearrow1$ on $X$ as
$n\to \infty$, and $ \| |  \nabla (h_n u) | \|_{L^p(X)} \to \| |  \nabla  u | \|_{L^p(X)}$. Taking the limit in the inequality 
 \[
\| h_n u \|_{p,1/2} \le C \| |  \nabla (h_n u) | \|_{L^p(X)}
\]
yields the result.
\end{proof}

%%%%%%%%%%%%%%%%%%%%%%%%%%%%%%%%%%%%%%%%%%%%%%%%%%%%%%%%%%
\subsection{Continuity of $P_t$ in the Besov spaces and  critical exponents}

We first note the following straightforward continuity property of $P_t$ in the Besov spaces.

\begin{proposition}\label{continuity Pt}
Suppose that the quasi Bakry-\'Emery condition~\eqref{eq:strong-BE} holds. Let $p>1$. There exists a constant $C_p>0$ such that for every $f\in L^p(X, \mu)$ and $t>0$
\[
\Vert P_t f \Vert_{p,1/2} \le \frac{C_p}{t^{1/2}} \Vert f \Vert_{L^p(X)}.
\]
\end{proposition}
\begin{proof}
This is a consequence of Lemma \ref{lem:LpGradient} and Theorem \ref{thm:BesovUB}.
\end{proof}
\begin{remark}
The above result is true without the quasi Bakry-\'Emery condition for $1<p\le 2$ on very general Dirichlet spaces, see \cite[Theorem 5.1]{ABCRST1}.
\end{remark}

For $p \ge 1$, as in \cite{ABCRST1}, we define the $L^p$ Besov density critical exponent $\alpha_p^*(X)$ and triviality critical exponent $\alpha_p^\#(X)$ as follows:
\begin{equation*}
\alpha_p^*(X) = \sup \{ \alpha>0\,:\, \mathbf{B}^{p,\alpha}(X) \text{ is dense in } L^p(X)\},
\end{equation*}
\[
\alpha^\#_p(X)=\sup \{ \alpha >0\, :\, \mathbf{B}^{p,\alpha}(X) \text{ contains non-constant functions} \}.
\]
%We have the following result:

\begin{theorem}
Suppose that the weak Bakry-\'Emery condition~\eqref{eq:weak-BE} holds, then for $1 \le p \le 2$,
\[
\alpha_p^*(X)=\alpha_p^\#(X)=\frac{1}{2}.
\]
Furthermore, if the quasi Bakry-\'Emery condition~\eqref{eq:strong-BE} holds, then for every $p > 2$,
\[
\alpha_p^*(X)=\alpha_p^\#(X)=\frac{1}{2}.
\]
\end{theorem}

\begin{proof}
Assume that the weak Bakry-\'Emery condition~\eqref{eq:weak-BE} holds and begin with the case $p=1$.   Let $f \in \mathbf{B}^{1,\alpha}(X)$ with $\alpha >1/2$. Since $\mathbf{B}^{1,\alpha}(X) \subset \mathbf{B}^{1,1/2}(X)=BV(X)$, we deduce that $f$ is a BV function. Now since $f \in \mathbf{B}^{1,\alpha}(X)$, one has for every $ t >0$,
\[
\int_X \int_X p_t(x,y) |f(x)-f(y)| d\mu(x) d\mu(y) \le t^\alpha \| f \|_{1,\alpha}.
\]
By using the gaussian heat kernel lower bound we obtain
\[
\liminf_{\varepsilon\to 0^+}\frac{1}{\varepsilon}
 \iint_{\Delta_\varepsilon}\frac{|f(y)-f(x)|}{\mu(B(x,\varepsilon))} d\mu(x) d\mu(y) =0,
 \]
so $\| Df\|(X)=0$, and from Remark~\ref{constancy 1} one gets that $f$ is constant. It follows that $\alpha^\#_1(X)\le 1/2.$ On the other hand, from Corollary 4.8 in \cite{ABCRST1}, $\mathbf{B}^{1,1/2}(X)$ is dense in $L^1(X)$, so $\alpha_1^*(X)=\alpha_1^\#(X)=\frac{1}{2}$.
From Proposition 5.6 in \cite{ABCRST1}, one has:
\begin{enumerate}
\item  Both $p\mapsto \alpha_p^* (X)$ and $p\mapsto \alpha_p^\# (X)$ are non-increasing;
\item For $1\leq p \leq 2$ we have $\alpha_p^\# (X)\geq \alpha_p^*(X) \geq \frac{1}{2}$.
\end{enumerate}
Therefore, for $1 \le p \le 2$ we also have $\alpha_p^*(X)=\alpha_p^\#(X)=\frac{1}{2}$.

Now let $p>2$ and assume the quasi Bakry-\'Emery condition~\eqref{eq:strong-BE}. In that case, according to Proposition \ref{continuity Pt}, for every $f \in L^p(X)$ and $t>0$ one has $P_t f \in \mathbf{B}^{p,1/2}(X)$. Thus, $\mathbf{B}^{p,1/2}(X)$ is dense in $L^p(X)$ by strong continuity of the semigroup $P_t$ in $L^p(X)$. Hence $\alpha_p^*(X) \ge 1/2$. Using again the fact that both $p\mapsto \alpha_p^* (X)$ and $p\mapsto \alpha_p^\# (X)$ are non-increasing and moreover that $\alpha_2^*(X)=\alpha_2^\#(X)=\frac{1}{2}$, one concludes that for  every $p > 2$, $\alpha_p^*(X)=\alpha_p^\#(X)=\frac{1}{2}.$
\end{proof}

\section{Sobolev and isoperimetric inequalities}\label{section Sobolev local}

Combining the conclusions in this paper with the results in \cite[Section 6]{ABCRST1}, we immediately obtain the 
following  results that generalize the Sobolev embedding theorems from the classical Euclidean setting (see for example~\cite{MazyaSobolev})
and metric upper gradient setting (see for example~\cite{HKST} and \cite{HK00})
to the setting of Dirichlet forms and BV functions.

%many of the known results found in the literature.
The following proposition is a weak-type version of the standard Sobolev embedding theorem. It gives
weak-$L^q$ control of the Besov function $f$, with $q$ the Sobolev conjugate of $p$, and can therefore be used to control the $L^s$-norm of $f$ in terms of the Besov norm of $f$ when $1\le s<pQ/(Q-p)$.

\begin{proposition}\label{sobofrac}
If  the volume growth condition 
$\mu(B(x,r)) \ge C_1 r^Q$,  $r \ge 0$,  is satisfied for some $Q>0$
then one has the following weak type Besov space embedding. Let $0<\delta < Q $ and $1 \le p < \frac{Q}{\delta} $.   Then
 there exists a constant $C_{p,\delta} >0$ such that for every $f \in \mathbf{B}^{p,\delta/2}(X) $,
\[
\sup_{s \ge 0} s \mu \left( \{ x \in X, | f(x) | \ge s \} \right)^{\frac{1}{q}} 
\le C_{p,\delta} \sup_{r>0} \frac{1}{r^{\delta+Q/p}}\biggl(\iint_{\{(x,y)\in X\times X\mid d(x,y)<r\}}|f(x)-f(y)|^{p}\,d\mu(x)\,d\mu(y)\biggr)^{\frac1p},
\]
where $q=\frac{p Q}{ Q -p \delta}$. Furthermore, for every $0<\delta <Q $, there exists a constant 
$C_{\emph{iso},\delta}$ such that for every measurable $E \subset X$, $\mu(E) <+\infty$,
\begin{align}\label{frac-iso}
\mu(E)^{\frac{Q-\delta}{Q}} 
\le C_{\emph{iso},\delta} \sup_{r>0} \frac{1}{r^{\delta+Q}} (\mu \otimes \mu) \left\{ (x,y) \in E \times E^c\, :\, d(x,y) \le r\right\} 
\end{align}
\end{proposition}

\begin{proof}
From  the heat kernel upper bound \eqref{eq:heat-Gauss}, the  volume growth condition $\mu(B(x,r)) \ge C_1 r^Q$,  $r \ge 0$, 
implies the ultracontractive estimate
\begin{equation}\label{eq:ultracontractive}
p_t(x,y) \le \frac{C}{t^{Q/2}}.
\end{equation}
We are therefore in the framework of  Theorem 6.1 in \cite{ABCRST1}, from which one obtains that there is a constant $C_{p,\delta} >0$ such that for every $f \in \mathbf{B}^{p,\delta/2}(X) $,
\[
\sup_{s \ge 0}\, s\, \mu \left( \{ x \in X\, :\, | f(x) | \ge s \} \right)^{\frac{1}{q}} \le C_{p,\delta}  \| f \|_{p,\delta/2}
\]
where $q=\frac{p Q}{ Q -p \delta}$. The conclusion follows from Theorem~\ref{prop:B=B}.
\end{proof}

\begin{example}
Assume that $X=\mathbb{R}^d$ is equipped with the standard Dirichlet form and the Lebesgue measure $\lambda^d$. 
If $E$ is a Borel set whose boundary $\partial E\subset \mathbb{R}^d$ is closed and $m$-rectifiable,  by~\cite[Theorem 3.2.39]{Fed69} we have
\begin{equation}
\limsup_{r\to 0+}\frac{1}{r^{d-m}} \lambda^d((\partial E)_r)= \frac{2\lambda^{m}(\partial E)\Gamma\big(\frac{1}{2}\big)^m}{m\Gamma\big(\frac{m}{2}\big)},
\end{equation}
where $(\partial E)_r$ denotes the $r$-neighborhood of $\partial E$. This implies $\mathbf{1}_E \in \mathbf{B}^{1,\frac{d-m}{2}}(\mathbb{R}^d)$ and proposition \ref{sobofrac}  \eqref{frac-iso},  is satisfied with $Q=d$, $\delta=d-m$. For instance, if $E$ is the so-called Koch snowflake domain in $\mathbb R^2$ then $d=2$ and $m=\frac{\log4}{\log3}$.
\end{example}

In Euclidean space there is a standard method for using the above weak-type Sobolev embedding to obtain the usual Sobolev embedding theorem, in which  the weak-$L^q$ control of $f$ is replaced by the strong-$L^q$ control.  However this approach uses locality properties which need not be valid for the Besov seminorm $\|\cdot\|_{p,\alpha}$. We direct
the interested reader to~\cite{HK00} for more details on this topic.

%In the case where $\delta=1/2$ and the weak Bakry-\'Emery estimate is satisfied, then one gets a strong Sobolev inequality.

The one circumstance we have investigated in which the Besov seminorm has a locality property arose in Theorem~\ref{thm:W=BV}, see also Remark~\ref{chaining metric}, for the space $\mathbf{B}^{1,1/2}$ under the assumption of a weak Bakry-\'Emery estimate, in which case we had $\mathbf{B}^{1,1/2}=BV(X)$.  This locality property lets us obtain a standard Sobolev embedding in which the $L^q$ norm is controlled by the BV norm.  We may view this as an extension of  known results on Riemannian manifolds with non-negative Ricci curvature (see Theorem~8.4 in~\cite{Ledoux96}) or on Carnot groups (see~\cite{Varopoulos89}) to our metric measure Dirichlet setting under the further hypothesis that there is a weak Bakry-\'Emery estimate.

\begin{theorem}
Suppose that the weak Bakry-\'Emery estimate~\eqref{eq:weak-BE} is satisfied. If the volume growth condition 
$\mu(B(x,r)) \ge C_1 r^Q$,  $r \ge 0$,  is satisfied for some $Q>0$, then there exists a constant $C_2 >0$ such 
that for every $f \in BV(X)$,
\[
\| f \|_{L^q(X)} \le C_2 \| Df \|(X)
\]
where $q=\frac{Q}{ Q-1}$. In particular, if $E$ is a set with finite perimeter in $X$, then
\[
\mu(E)^{\frac{Q-1}{Q}} \le C_2 P(E,X).
\]
\end{theorem}

\begin{proof}
Observe that as in the above proof, the heat kernel satisfies the ultracontractive estimate~\eqref{eq:ultracontractive}.
From  Theorem~\ref{thm:W=BV} we have
\[
\|f\|_{1,1/2}  \le C\liminf_{s \to 0}  s^{-1/2}  \int_X P_s (|f-f(y)|)(y) d\mu(y).
\]
This verifies a condition denoted by $(P_{1,1/2})$ in Definition~6.7 of~\cite{ABCRST1}), putting us in the framework of~\cite[Theorem 6.9]{ABCRST1} with $p=1$, $\alpha=1/2$ and $\beta=Q/2$. Notice also that $\|f\|_{1,1/2}\le C\|Df\|(X)$ from Theorem~\ref{thm:W=BV}, so we have
\[
\| f \|_{L^q(X)} \le C\|f\|_{1,1/2} \le C_2 \| Df \|(X),
\] 
where $q=\frac{Q}{ Q-1}$. Taking $f=\mathbf 1_E$ then yields
\begin{equation*}
\mu(E)^{\frac{Q-1}{Q}} \le C_2 P(E,X).\qedhere
\end{equation*}
%{\bf What is this $(P_{1,1/2})$? Needs correct reference rather than short-hand citations. Need to fix this. Nages}
%This follows from
%\[
%\sup_{s >0 }  s^{-1/2}  \int_X P_s (|f-f(y)|)(y) d\mu(y)  \le C\liminf_{s \to 0}  s^{-1/2}  \int_X P_s (|f-f(y)|)(y) d\mu(y) ,
%\]
%which is a consequence from Theorem \ref{thm:W=BV}.
\end{proof}

\bibliographystyle{plain}
 \bibliography{BV_Refs}

\end{document}